\documentclass[10pt]{article}

\usepackage[margin=1in]{geometry}
\usepackage{graphicx}
\usepackage{amsmath}
\usepackage{amsfonts}
\usepackage{amsthm}
\usepackage{xcolor}
\usepackage[bookmarksopen=true]{hyperref}

\usepackage{marginnote}

\numberwithin{equation}{section}

\newtheorem{thm}{Theorem}[section]
\newtheorem{claim}[thm]{Claim}
\newtheorem{lem}[thm]{Lemma}
\newtheorem{prop}[thm]{Proposition}
\newtheorem{cor}[thm]{Corollary}

\newtheorem{mydef}{Definition}[section]

\newcommand{\R}{\mathbb{R}}

\newcommand{\B}{\mathbb{B}}
\newcommand{\N}{\mathbb{N}}

\newcommand{\A}{\mathcal{A}}

\begin{document}

\title{On the well-posedness and asymptotic behavior of the generalized KdV-Burgers equation}

\nocite{*}

\author{F. A. Gallego and A. F. Pazoto\footnote{Institute of Mathematics, Federal University of Rio de Janeiro, UFRJ, P.O. Box 68530, CEP 21945-970, Rio de Janeiro, RJ, Brasil. E-mail: fgallego@ufrj.br, ademir@im.ufrj.br. }
}

\date{}
\maketitle

\begin{abstract}
In this paper we are concerned with the well-posedness and the exponential stabilization of the generalized Korteweg-de Vries Burgers equation, posed on the whole real line, under the effect of a damping term.
Both problems are investigated when the exponent $p$ in the nonlinear term ranges over the interval $[1, 5)$. We first prove the global well-posedness in $H^s(\R)$, for $0\leq s \leq 3$ and  $1 \leq p <2$, and in $H^3(\R)$, when $p\geq 2$. For $2 \leq p <5$,  we prove the existence of global solutions in the $L^2$-setting. Then, by using multiplier techniques combined with interpolation theory, the exponential stabilization is obtained for a indefinite damping term and $1 \leq p <2$ . Under the effect of a localized damping term the result is obtained when $2 \leq p <5$. Combining multiplier techniques and compactness arguments it is shown that the problem of exponential decay is reduced to prove the unique continuation property of weak solutions\let\thefootnote\relax\footnote{FAG was supported by CAPES (Brazil) and AFP was partially supported by CNPq (Brazil).}
\\
\vspace{1mm}


\end{abstract}

\section{Introduction}
It is common knowledge that many physical problems, such as nonlinear shallow-water waves and wave motion in plasmas can be described by the family of the Korteweg-de Vries (KdV) equation. The KdV-type equations have also been used to describe a wide range of important physical phenomena
related to acoustic waves in a harmonic crystal, quantum field theory, plasma physics and solid-state physics. In what concerns the study of wave propagation in a tube filled with viscous fluid or flow of the fluid containing gas bubbles, for example, the control equation can be reduced to the so-called KdV-Burgers equation \cite{su1969korteweg}. It is commonly obtained from the KdV equation by adding a viscous term and combines nonlinearity, linear dissipation and dispersion terms:
\[
u_t + \delta u_{xxx}-  \nu u_{xx} + uu_x = 0,\quad  t > 0,\, x \in \R.
\]
Since $\delta$ and $\nu$ are positive numbers, the model can be viewed as a generalization of the KdV and Burgers equation.
Particularly, the Burgers equation is a simple model equation for a variety of diffusion/dissipative processes in convection dominated systems, which include formation of weak shocks, traffic flow, turbulence, etc. If besides convective nonlinearity and dissipation/diffusion mechanism, the dispersion also plays its role over the spatial and temporal scales of interest, then the simplest nonlinear PDE governing the wave dynamics is the combination of both KdV and Burgers equation which is known as KdV-Burgers equation.
%In the long wavelength limit, KdVB equation has been derived to describe the wave dynamics for numerous fluid and plasma systems.

In this work we are concerned with the generalized KdV-Burgers equation (GKdV-B) under the effect of a damping term represented by a function
$b = b(x)$, more precisely,
\begin{equation}\label{e1}
\left\lbrace \begin{tabular}{l l}
$u_t +u_{xxx}-u_{xx}+a(u)u_x+b(x)u=0$ & in $\R\times\R_+$ \\
$u(x,0)=u_0(x)$ & in $\R.$
\end{tabular}\right.
\end{equation}
Our main purpose is to address two mathematical issues connected to the initial value problem \eqref{e1}: global well-posedness and large-time behavior of solutions. More precise, we establish the well-posedness and the exponential decay of solutions in the classical Sobolev spaces $H^s$. Therefore, as usual, let us first consider the energy associated to the model, given by
$$E(t)=\frac{1}{2}\int_{\mathbb{\R}}u^2(x,t) dx.$$
Thus, at least formally, the solutions of \eqref{e1} should satisfy
\begin{equation}\label{diss-energy}
\displaystyle \frac{d}{dt}E(t) = - \int_{\mathbb{\R}}u^2_x dx - \int_{\mathbb{\R}} b(x) u^2  dx,
\end{equation}
for any positive $t$. Then, if we assume that $b(x)\geq b_0$, for some $b_0 >0$, it is forward to infer that $E(t)$ converges to zero exponentially. By contrast, when the damping function $b$ is allowed to change sign or is effective only a subset of the domain, the problem is much more subtle. Moreover,  whether \eqref{diss-energy}  generates a flow that can be continued indefinitely in the temporal variable, defining a solution valid for all $t\geq 0$, is a nontrivial question.

In order to provide the tools to handle with both problems, we assume that $a = a(x)$ is a positive real-valued function that satisfies the growth conditions
\begin{equation}\label{e2}
\left\lbrace \begin{tabular}{l}
\vspace{1mm} $|a^{(i)}(\mu)| \leq C(1+|\mu|^{p-j}), \quad \forall \mu \in \R$, \mbox{ for some } $C > 0 $,\\
$j=0,1$ \mbox{ if } $1\leq p < 2$ and $j=0,1,2$ \mbox{ if } $p\geq 2$.
\end{tabular} \right.
\end{equation}
Moreover, in order to obtain the exponential stability in the case $1\leq p <2$, we take an indefinite damping satisfying
\begin{equation}\label{hyp b}
\left\lbrace \begin{tabular}{l}
$b \in H^1(\R)$  and  $b(x) \geq \lambda_0 +\lambda_1(x)$,  almost everywhere, for some $\lambda_0 > 0$ and  $\lambda_1 \in L^p(\R)$,\\
such that $\|\lambda_1\|_{L^p(\R)} < \left(\frac{\lambda_0}{c_p}\right)^{1-\frac{1}{2p}},$ where $c_p=\left(1-\frac{1}{2p}\right)\left(\frac{2}{p}\right)^{\frac{1}{2p-1}}.$
\end{tabular}
\right.
\end{equation}
Concerning the case $p\geq 2$, we consider a localized damping which acts only a bounded subset of the line, more precisely,
\begin{equation}\label{hyp a}
%\left\lbrace
\begin{tabular}{l}
$b \in H^1(\R)$ is nonnegative and  $b(x)\geq \lambda_0 >0$ almost everywhere in $(-\infty,\alpha)\cup(\beta,\infty)$, for some $\alpha, \beta\in \R$.
\end{tabular}
%\right.
\end{equation}

Our analysis was inspired by the results obtained by Cavalcanti et al. for KdV-Burgers equation \cite{cavalcanti2014global} and by Rosier and Zhang for the generalized KdV equation posed on a bounded domain \cite{rosier2006global} (see also \cite{linares2007exponential}). In this context, we refer to the survey \cite{rosier2009control} for a quite complete review on the state of art.

When $1 \leq p <2$ and $0\leq s\leq 3$, we obtain the global well-posedness in the class $B_{s,T}=C([0,T];H^s(\R))\cap L^2(0,T;H^{s+1}(\R))$  and prove that the solutions decays exponentially to zero in $H^s(\R)$, where $H^s$ denotes the classical Sobolev spaces. As it is known in the theory of dispersive wave equations, the results depend on the local theory, on the a priori estimates satisfied by the solutions and also on linear theory. Indeed, we combine the Duhamel formula and a contraction-mapping principle to prove directly the local well-posedness. In order to get the global result we derive energy-type inequalities and make use of interpolation arguments. Those a priori estimates are sufficient to yield the global stabilization result and a strong smoothing property for solutions; $u\in C([\varepsilon,T];H^s(\R))\cap L^2(\varepsilon,T;H^{s+1}(\R))$, for any $\varepsilon >0$. Our analysis extends the results obtained in \cite{cavalcanti2014global} from which we borrow some ideas involved in our proofs. When $p\geq 2$ we can use the same approach to prove that the global well-posedness also holds in $B_{3,T}$. In order to get the result in a stronger/weaker norm,  we need a priori global estimates. However, the only available priori estimate for \eqref{e1} is the estimate provided by \eqref{diss-energy}, which does not guarantee existence of global in time solutions. In fact, we do not know if the problem is locally well-posed in the energy space. Therefore, we restrict ourselves to the case $2\leq p <5$ to prove that the estimate provided by the energy dissipation law holds and establish the existence of global solutions in the space $C_{\omega}([0,T];L^2(\R))\cap L^2(0,T;H^1(\R))$. The uniqueness remains an open problem. In what concerns the asymptotic behavior,  we prove the exponential decay in the $L^2$-setting by following the approach used in \cite{rosier2006global}. It combines multiplier techniques and compactness arguments to reduce the problem to some unique continuation property for weak solutions. To overcome this problem we develop a Carleman inequality by modifying (slightly) a Carleman estimate obtained by Rosier in \cite{rosier2000exact} to study the controllability properties of the KdV equation. It allows us to prove unique continuation property directly.

The program of this work was carried out for the particular choice of damping effect appearing in \eqref{e1} and aims to establish as a fact that such a model predicts the interesting qualitative properties initially observed for the KdB-Burgers type equations. Consideration of this issue for nonlinear dispersive equations has received considerable attention, specially the problems on the time decay rate. At this respect, it is important to point out that the approach used here was successfully applied in the context of the KdV equation, posed on $\R^+$ and $\R$, under the effect of a localized damping term \cite{cavalcanti2012decay,linares2009asymptotic}.
%The analysis developed in those papers was strongly motivated by the results obtained for the KdV equation posed on a bounded interval (see \cite{rosier2009control} for a quite complete review on the state of art).
We also remark that, in the absence of the damping term $b$, the stabilization problem was addressed by Bona and Luo \cite{bona1995more,bona1993decay}, complementing the earlier studies developed in \cite{amick1989decay,biler1984asymptotic,dix1992dissipation} and  deriving sharp polynomial decay rates for the solutions. Later on, in \cite{bona2001asymptotic,kang2014simple}, the authors improved upon the foregoing theory. The asymptotic behavior was also discussed in the language of the global attractors \cite{dlotko2011generalized,dlotko2010asymptotic}. More precisely, the authors study the large time behaviour of the corresponding semigroup on constructing a global attractor.

The analysis described above was organized in two sections. In Section 2 we establish the global well-posedness results. Section 3 is devoted to the stabilization problem.
Finally, in the Appendix, we prove a Carleman inequality. In all sections we split the results into several steps in order to make the reading easier.

\section{Well-posedness.}
First we consider the corresponding linear inhomogeneous initial value problem,
\begin{equation}\label{ee1}
\left\lbrace \begin{tabular}{l l}
$u_t-u_{xx} +u_{xxx}+b(x)u=f$ & $(x,t)\in \R\times\R_+$ \\
$u(x,0)=u_0(x)$ & $x\in \R$.
\end{tabular}\right.
\end{equation}
Setting
\begin{equation*}
\text{$A_b :=\partial_x^2 -\partial_x^3-bI$ and  $D(A_b)=H^3(\R), b \in L^{\infty}(\R)$}
\end{equation*}
(\ref{ee1}) can be written in the form
\begin{equation*}
\left\lbrace  \begin{tabular}{l}
$u_t = A_bu + f$ \\
$u(0) = u_0$.
\end{tabular}\right.
\end{equation*}
According to \cite{cavalcanti2014global}, $A_b$ generates a strongly continuous semigroup $\{S(t)\}_{t \geq 0}$ of contractions in $L^2(\R)$. Hence,
%for any given $u_0 \in L^2(\R)$, $T > 0$ and $f \in L^1(0, T ;L^2(\R))$, (\ref{ee1}) has a unique mild solution $u \in C([0, T ];L^2(R))$, such that
%\begin{equation*}
%\|u\|_{C([0, T ];L^2(\R))} :=\sup_{t\in[0,T]}{\|u(t)\|_{L^2(\R)}} \leq\|u_0\|_{L^2(\R)}+\|f\|_{L^1(0, T;L^2(\R))}
%\end{equation*}
if we consider the Banach space
\begin{equation}\label{e3}
\left\lbrace\begin{tabular}{l}
$B_{s,T}:=C([0,T];H^s(\R))\cap L^2(0,T;H^{s+1}(\R))$ \\
\\
$\|u\|_{s,T}=\sup_{t\in[0,T]}{\|u(t)\|_{H^s(\R)}} +\|\partial^{s+1}_{x}u\|_{L^2(0,T;L^2(\R))}$,
\end{tabular}\right.
\end{equation}
%Note that $B_{s,T} \hookrightarrow B_{r,T}$ for all $r\leq s$. Let $u_0 \in L^2(\R)$, by a mild  solution of (\ref{e1}) we mean a function $u \in B_{0,T}$ , $T > 0$, satisfying
%\begin{equation}\label{mild}
%u(t)=S(t)u_0 + \int_0^t S(t - s)a(u(s))u_x(s)ds, \quad t \in [0, T ].
%\end{equation}
%By a global mild solution of (\ref{e1}) we mean a function $u:[0,\infty) \rightarrow  H^1(\R)$ whose restriction to every bounded interval $[0, T]$ is a mild solution of (\ref{e1}).
the following result holds:

\begin{prop}\label{prop5}
Let $T>0$. If $u_0 \in L^2(\R)$ and $f \in L^1(0,T;L^2(\R))$, (\ref{ee1}) has a unique mild solution $u \in B_{0,T}$, and
\begin{equation*}
\|u\|_{0,T}\leq C_T\left\lbrace \|u_0\|_2+\|f\|_{L^1(0,T;L^2(\R))}\right\rbrace, \quad \text{with} \quad C_T=2e^{T\|b\|_{\infty}}.
\end{equation*}
Furthermore, the following energy identity holds for all $t \in [0,T]$:
\begin{equation}\label{e116}
\|u(t)\|_2^2+2\int_0^t\|u_x(s)\|_2^2ds+2\int_0^t\int_{\R}b(x)\|u(x,s)\|_2^2dxds=\|u_0\|_2^2+2\int_0^t\int_{\R}f(x,s)u(x,s)dxds.
\end{equation}
\end{prop}

\begin{proof}
See \cite[Proposition 4.1]{cavalcanti2014global}.
\end{proof}

\subsection{\texorpdfstring{Case $1 \leq p <2$.}{Case 1}}
In order to establish the well-posedness of (\ref{e1}) we need the following technical Lemmas, that will play an important role in the proofs:

\begin{lem}[Generalized H\"{o}lder inequality]\label{hg}
Suppose that for $i=1,2,...,n$, $f_i  \in L^{p_i}$ and $\displaystyle\sum_{i=1}^n \frac{1}{p _i}=1$. Then,
\begin{align}\label{hi}
\|f_1.f_2...f_n\|_{L^1} \leq \sum_{i=1}^n \|f_i\|_{L^{p_i}}.
\end{align}

\end{lem}

\begin{lem}\label{lm1}
Let $a\in C^0(\R)$ be a function satisfying
\begin{gather}\label{e4}
|a(\mu)| \leq C(1+|\mu|^{p}), \quad \forall \mu \in \R,
\end{gather}
with $0\leq p < 2$. Then, there exists a positive constant $C$, such that, for any $T>0$ and $u,v \in B_{0,T}$, we have
\begin{equation*}
\|a(u)v_x\|_{L^1(0,T;L^2(\R))}\leq 2^{\frac{p}{2}}CT^{\frac{2-p}{4}}\|u\|_{0,T}^p\|v\|_{0,T}+CT^{\frac{1}{2}}\|v\|_{0,T}.
\end{equation*}
\end{lem}

\begin{proof}
Recall that $H^1(\R) \hookrightarrow L^{\infty}(\R)$ and
\begin{equation}\label{e6}
\|u\|_{\infty}^2\leq 2\|u\|_2\|u_x\|_2,
\end{equation}
for all $u \in H^1(\R)$. %Indeed, if $u \in C_c^{\infty}(\R)$ and $y \in \R$,
%\begin{equation*}
%|u(y)^2|= \left|\int_{-\infty}^{y} 2uu_xdx\right|\leq 2 \int_{-\infty}^{y} |u||u_x|dx\leq 2\|u\|_2\|u_x\|_2
%\end{equation*}
%proving the estimate for smooth functions, the general case follows by density.
On the other hand, by (\ref{e4}),
\begin{align*}
\|a(u)v_x\|_{L^1(0,T;L^2(\R))}&\leq C \int _{0}^{T}\|\left(1+|u(t)|^p\right)v_x(t)\|_2dt \\
&\leq C \int _{0}^{T}\|v_x(t)\|_2dt+ C\int _{0}^{T}\|u(t)\|_{\infty}^p\|v_x(t)\|_2dt.
\end{align*}
Using H\"{o}lder inequality $(\ref{hi})$ and $(\ref{e6})$, we have
\begin{align*}
\|a(u)v_x\|_{L^1(0,T;L^2(\R))}&\leq C T^{\frac{1}{2}}\|v_x\|_{L^2(0,T;L^2)}+ 2^{\frac{p}{2}}C\int _{0}^{T}\|u(t)\|_{2}^{\frac{p}{2}}\|u_x(t)\|_{2}^{\frac{p}{2}}\|v_x(t)\|_2dt \\
&\leq  C T^{\frac{1}{2}}\|v_x\|_{L^2(0,T;L^2)}+ 2^{\frac{p}{2}}C\|u\|_{C([0,T];L^2)}^{\frac{p}{2}}\int _{0}^{T}\|u_x(t)\|_{2}^{\frac{p}{2}}\|v_x(t)\|_2dt.
\end{align*}
Applying Lemma \ref{hg} with $\frac{p}{4},\frac{2-p}{4}$ and $\frac{1}{2}$, it follows that
\begin{align*}
\|a(u)v_x\|_{L^1(0,T;L^2(\R))}&\leq C T^{\frac{1}{2}}\|v\|_{0,T}+ 2^{\frac{p}{2}}CT^{\frac{2-p}{4}}\|u\|_{0,T}^{\frac{p}{2}}\|u_x\|_{L^2(0,T;L^2(\R))}^{\frac{p}{2}}\|v_x\|_{L^2(0,T;L^2(\R))} \\
&\leq  2^{\frac{p}{2}}CT^{\frac{2-p}{4}}\|u\|_{0,T}^p\|v\|_{0,T}+CT^{\frac{1}{2}}\|v\|_{0,T}.
\end{align*}
\end{proof}

\begin{lem}\label{lm2}
For any $T>0$, $b \in L^{\infty}(\R)$ and $u, v, w \in B_{0,T}$, we have
\begin{enumerate}
\item[(i)] $\|bu\|_{L^1(0,T;L^2(\R))} \leq T^{\frac{1}{2}}\|b\|_{\infty}\|u\|_{0,T}$,
\item[(ii)] $\|uw_x\|_{L^1(0,T;L^2(\R))} \leq 2^{\frac{1}
{2}}T^{\frac{1}{4}}\|u\|_{0,T}\|w\|_{0,T}$.
\end{enumerate}
If $1\leq p <2$,
\begin{enumerate}
%\item[(iii)] $\|u|w|^{p-1}w_x\|_{L^1(0,T;L^2(\R))} \leq 2^{\frac{p}{2}}T^{\frac{2-p}{4}}\|u\|_{0,T}\|w\|_{0,T}^p$
\item[(iii)] $\|u|v|^{p-1}w_x\|_{L^1(0,T;L^2(\R))} \leq 2^{\frac{p}{2}}T^{\frac{2-p}{4}}\|u\|_{0,T}\|w\|_{0,T}\|v\|_{0,T}^{p-1}$,
\item[(iv)] Consider the map $M: B_{0,T} \rightarrow L^1(0,T;L^2(\R))$ defined by $Mu:=a(u)u_x$. Then, $M$ is locally Lipschitz continuous and
\begin{align*}
\|Mu-Mv\| |_{L^1(0,T;L^2(\R))} \leq C&\left\lbrace 2^{\frac{1}{2}}T^{\frac{1}{4}}\|u\|_{0,T} +2^{\frac{p}{2}}T^{\frac{2-p}{4}}\left(\|u\|_{0,T}^p+\|u\|_{0,T}\|v\|_{0,T}^{p-1} \right. \right.\\
&\left.\left.+\|v\|_{0,T}^{p}\right) +T^{\frac{1}{2}}\right\rbrace\|u-v\|_{0,T},
\end{align*}
where $C$ is a positive constant.
\end{enumerate}
\end{lem}

\begin{proof}
\begin{enumerate}
\item[(i)] Using H\"{o}lder inequality, we have
\begin{align*}
\|bu\|_{L^1(0,T;L^2(\R))} \leq T^{\frac{1}{2}} \|b\|_{\infty}\|u\|_{L^2(0,T;L^2)} \leq T^{\frac{1}{2}} \|b\|_{\infty}\|u\|_{0,T}.
\end{align*}
\item[(ii)] Combining ($\ref{e6}$) %(i.e $H^1(\R) \hookrightarrow L^{\infty}(\R)$)
and Lemma \ref{hg} with $\frac{1}{2}, \frac{1}{4}$ and $\frac{1}{4}$, it follows that
\begin{align*}
\|uw_x\|_{L^1(0,T;L^2(\R))} &\leq \int_0^T \|u(t)\|_{\infty}\|w_x(t)\|_2dt \leq 2^{\frac{1}{2}}\int_0^T\|u(t)\|_{2}^{\frac{1}{2}}\|u_x(t)\|_{2}^{\frac{1}{2}}\|w_x(t)\|_2dt \\
&\leq 2^{\frac{1}{2}}\|u(t)\|_{C([0,T];L^2)}^{\frac{1}{2}}\left(\int_0^T\|u_x(t)\|_{2}^{2}dt\right)^{\frac{1}{4}}\left(\int_0^T\|w_x(t)\|_{2}^{2}dt\right)^{\frac{1}{2}}T^{\frac{1}{4}} \\
&\leq 2^{\frac{1}{2}}T^{\frac{1}{4}}\|u\|_{0,T}\|w_x\|_{0,T}.
\end{align*}
\item[(iii)] We proceed as in (i) combining ($\ref{e6}$) and Lemma \ref{hg} with $\frac{1}{4}, \frac{p-1}{4}, \frac{2-p}{4}$ and $\frac{1}{2}$ to obtain
\begin{align*}
\|u|v|^{p-1}w_x\|_{L^1(0,T;L^2(\R))} &\leq \int_0^T \|u(t)\|_{\infty}\|v(t)\|_{\infty}^{p-1}\|w_x(t)\|_2dt \\
&\leq 2^{\frac{p}{2}}\int_0^T \|u(t)\|_{2}^{\frac{1}{2}}\|u_x(t)\|_{2}^{\frac{1}{2}}\|v(t)\|_{2}^{\frac{p-1}{2}}\|v_x(t)\|_{2}^{\frac{p-1}{2}}\|w_x(t)\|_2dt \\
&\leq 2^{\frac{p}{2}}\|u\|_{0,T}^{\frac{1}{2}}\|v\|_{0,T}^{\frac{p-1}{2}}\int_0^T \|u_x(t)\|_{2}^{\frac{1}{2}}\|v_x(t)\|_{2}^{\frac{p-1}{2}}\|w_x(t)\|_{2}dt \\
&\leq 2^{\frac{p}{2}}\|u\|_{0,T}^{\frac{1}{2}}\|v\|_{0,T}^{\frac{p-1}{2}}\left(\int_0^T \|u_x\|_{2}^{2}dt\right)^{\frac{1}{4}}\left(\int_0^T\|v_x\|_{2}^{2}dt\right)^{\frac{p-1}{4}}\left(\int_0^T\|w_x\|_{2}^{2}dt\right)^{\frac{1}{2}}T^{\frac{2-p}{4}} \\
&\leq 2^{\frac{p}{2}}T^{\frac{2-p}{4}}\|u\|_{0,T}^{\frac{1}{2}}\|v\|_{0,T}^{\frac{p-1}{2}}\|u\|_{0,T}^{\frac{1}{2}}\|v\|_{0,T}^{\frac{p-1}{2}}\|w\|_{0,T},
\end{align*}
which allows us to conclude the result.
\item[(iv)] Note that
\begin{equation*}
\|Mu-Mv\|_{L^1(0,T;L^2(\R))} \leq \|(a(u)-a(v))u_x\|_{L^1(0,T;L^2(\R))} +\|a(v)(u-v)_x\|_{L^1(0,T;L^2(\R))}.
\end{equation*}
Using the Mean Valued Theorem, $(ii)$, $(iii)$ and Lemma \ref{lm1}, we have
\begin{align*}
\|Mu -Mv \|_{0,T}\leq & C\|(1+|u|^{p-1}+|v|^{p-1})|u-v|u_x\|_{L^1(0,T;L^2)}+\|a(v)(u-v)_x\|_{L^1(0,T;L^2)} \\
\leq &C\left\lbrace 2^{\frac{1}{2}}T^{\frac{1}{4}}\|u-v\|_{0,T}\|u\|_{0,T} +2^{\frac{p}{2}}T^{\frac{2-p}{4}}\|u-v\|_{0,T}\|u\|_{0,T}^p\right. \\
&\left.+2^{\frac{p}{2}}T^{\frac{2-p}{4}}\|u-v\|_{0,T}\|u\|_{0,T}\|v\|_{0,T}^{p-1}
+2^{\frac{p}{2}}T^{\frac{2-p}{4}}\|u-v\|_{0,T}\|v\|_{0,T}^{p} +T^{\frac{1}{2}}\|u-v\|_{0,T}\right\rbrace.
\end{align*}
\end{enumerate}
\end{proof}

\noindent The above estimates lead to the following local existence result and a priori estimate:

\begin{prop}\label{r1}
Let $a$ be a function $C^1(\R)$ satisfying
\begin{gather*}
\text{$|a(\mu)| \leq C(1+|\mu|^{p})$ and $|a'(\mu)| \leq C(1+|\mu|^{p-1})$, $\forall \mu \in  \R$,}
\end{gather*}
with $1\leq p<2$. Let $b \in L^{\infty}(\R)$ and $u_0 \in L^2(\R)$. Then, there exist $T>0$ and a unique mild solution $u \in B_{0,T}$ of $(\ref{e1})$. Moreover,
\begin{equation}\label{e117}
\|u(t)\|_2^2+2\int_0^t\|u_x(s)\|_2^2ds+2\int_0^t\int_{\R}b(x)\|u(x,s)\|_2^2dxds=\|u_0\|_2^2, \qquad \forall t \in [0,T].
\end{equation}
\end{prop}

\begin{proof}
Let $T>0$ to be determined later.  For each $u \in B_{0,T}$  consider the problem
\begin{equation}\label{e9}
\left\lbrace \begin{tabular}{l}
$v_t = A_bv - Mu$ \\
$v(0)=u_0$,
\end{tabular} \right.
\end{equation}
where $A_bv=\partial_x^2v-\partial_x^3v-bv$ and $Mu=a(u)u_x$. Since $A_b$ generates a strongly continuous semigroup $\{S(t)\}_{t\geq 0}$ of contractions in $L^2(\R)$, Lemma \ref{lm1} and  Proposition \ref{prop5} allows us to conclude that (\ref{e9}) has a unique mild solution $v \in B_{0,T}$, such that
\begin{equation}\label{e10}
\|v\|_{0,T}\leq C_{T} \{ \|u_0\|_2+\|Mu\|_{L^1(0,T;L^2(\R))}\},
\end{equation}
where $C_{T}=2e^{T\|b\|_{\infty}}$. Thus, we can define the operator
\[
\text{$\Gamma:B_{0,T}\longrightarrow B_{0,T}$ given by $\Gamma (u) = v$.}
\]
By using Lemma $\ref{lm1}$ and $(\ref{e10})$, we have
\begin{equation*}
\|\Gamma u \|_{0,T}\leq C_T \{\|u_0\|_2+  2^{p/2}CT^{\frac{2-p}{4}}\|u\|_{0,T}^{p+1}+CT^{\frac{1}{2}}\|u\|_{0,T}\}.
\end{equation*}
Thus, for $u \in B_R(0):=\{u \in B_{0,T}: \|u\|_{B_{0,T}}\leq R\}$, it follows that
\begin{equation*}
\|\Gamma u \|_{0,T}\leq C_T \{\|u_0\|_2+  2^{p/2}CT^{\frac{2-p}{4}}R^{p+1}+CT^{\frac{1}{2}}R. \}
\end{equation*}
Choosing $R=2C_T\|u_0\|_2$, we obtain the following estimate
\begin{equation*}
\|\Gamma u \|_{0,T}\leq \left(K_1+\frac{1}{2}\right)R,
\end{equation*}
where $K_1=K_1(T)= 2^{p/2}C_TCT^{\frac{2-p}{4}}R^{p}+C_TCT^{\frac{1}{2}} .$ On the other hand, note that $\Gamma u - \Gamma w$ is solutions of
\begin{equation*}
\left\lbrace \begin{tabular}{l}
$v_t = A_bv - (Mu-Mw)$ \\
$v(0)=0$.
\end{tabular} \right.
\end{equation*}
Again, by applying Proposition \ref{prop5}, we have
\begin{align*}
\|\Gamma u -\Gamma w \|_{0,T}&\leq C_T\|Mu-Mw\|_{L^1(0,T;L^2)}
\end{align*}
and estimate $(iv)$ in Lemma \ref{lm2} allows us to conclude that
\begin{align*}
\|\Gamma u -\Gamma w \|_{0,T}\leq C_TC&\left\lbrace 2^{\frac{1}{2}}T^{\frac{1}{4}}\|u\|_{0,T} +2^{\frac{p}{2}}T^{\frac{2-p}{4}}\left(\|u\|_{0,T}^p+\|u\|_{0,T}\|w\|_{0,T}^{p-1} \right. \right. \notag\\
&\left.\left.+\|w\|_{0,T}^{p}\right) +T^{\frac{1}{2}}\right\rbrace\|u-w\|_{0,T}.
\end{align*}
Suppose that $u,w \in B_R(0)$ defined above. Then,
\begin{align*}
\|\Gamma u -\Gamma w \|_{B_{0,T}}&\leq K_2 \|u-w\|_{B_{0,T}},
\end{align*}
where $K_2=K_2(T)= C_TC\{2^{\frac{1}{2}}T^{\frac{1}{4}}R +3(2^{\frac{p}{2}})T^{\frac{2-p}{4}}R^{p}+T^{\frac{1}{2}} \}$. Since  $K_1\leq K_2$, we can  choose $T>0$ to obtain $K_2<\frac{1}{2}$ and
\begin{equation*}
\left\lbrace\begin{tabular}{l}
$\|\Gamma u\|_{B_{0,T}}\leq R$ \\
$\|\Gamma u-\Gamma w\|_{B_{0,T}}< \frac{1}{2}\|u-w\|_{B_{0,T}}$
\end{tabular}\right., \qquad \forall u,w \in B_R(0) \subset B_{0,T}.
\end{equation*}
Hence $\Gamma: B_R(0) \longrightarrow B_R(0)$ is a contraction and, by Banach fixed point theorem, we obtain a unique $u\in B_R(0)$, such that $\Gamma (u)=u$. Consequently, $u$ is a unique local mild solution of (\ref{e1}) and
\begin{equation}\label{e13}
\|u\|_{B_{0,T}}\leq 2C_T\|u_0\|_2.
\end{equation}
In order to prove (\ref{e117}) consider $v_n=\Gamma v_{n-1}$, $n\geq 1$. Since $\Gamma$ is a contraction, we have
\begin{align*}
\text{$\lim_{n\rightarrow \infty}v_n = u$ in $B_{0,T}$}.
\end{align*}
On the other hand, by (\ref{e116}) in Proposition \ref{prop5}, $v_n$ verifies the identity
\begin{equation*}
\|v_n(t)\|_2^2+2\int_0^t\|v_{nx}(s)\|_2^2ds+2\int_0^t\int_{\R}b(x)\|v(x,s)\|_2^2dxds=\|u_0\|_2^2+2\int_0^t\int_{\R}Mv_{n-1}(x,s)v_n(x,s)dxds.
\end{equation*}
Then, taking the limit as $n \rightarrow \infty$, we get
\begin{equation*}
\|u(t)\|_2^2+2\int_0^t\|u_{x}(s)\|_2^2ds+2\int_0^t\int_{\R}b(x)\|u(x,s)\|_2^2dxds=\|u_0\|_2^2
\end{equation*}
since the limit of the last term is $\int_0^t\int_{\R}Mu(x,s)u(x,s)dxds=0$. In fact, $\int_{\R}a(u(x))u_x(x)dx=\int_{\R}[A(u(x))]_xdx$ where $A(v)= \int_{0}^{v}a(s)ds$.
\end{proof}

\noindent From Proposition \ref{r1} we obtain our first global in time existence result:

\begin{thm}\label{teo1}
Let $a$ be a function $C^1(\R)$ satisfying
\begin{gather*}
\text{$|a(\mu)| \leq C(1+|\mu|^{p})$ and  $|a'(\mu)| \leq C(1+|\mu|^{p-1}), \quad \forall \mu \in \R,$}
\end{gather*}
with $1\leq p<2$. Let $b \in L^{\infty}(\R)$ and $u_0 \in L^2(\R)$. Then, there exist a unique global mild solution $u$ of $(\ref{e1})$, such that, for each $T>0$, there exist a nondecreasing continuous function $\beta_0: \R_+\rightarrow \R_+$ which satisfies
\begin{equation}\label{e8}
\|u\|_{0,T}\leq \beta_0(\|u_0\|_2)\|u_0\|_2.
\end{equation}
Moreover, the following energy identity holds for all $t\geq 0:$
\begin{equation}\label{e118}
\|u(t)\|_2^2+2\int_0^t\|u_x(s)\|_2^2ds+2\int_0^t\int_{\R}b(x)\|u(x,s)\|_2^2dxds=\|u_0\|_2^2.
\end{equation}
\end{thm}

\begin{proof}
By Proposition \ref{r1}, there exists a unique mild solution $u \in B_{0,T}$, for all $T < T_{max}\leq \infty$. Moreover,
\begin{equation*}
\|u\|_{0,T} \leq 4e^{\|b\|_{\infty}t}\|u_0\|_2, \qquad \forall t \in [0, T_{max}),
\end{equation*}
which implies that $u$ is a global mild solution of (\ref{e1}). On the other hand, (\ref{e13}) implies (\ref{e8})  with $\beta_0(s)=2C_T$. The identity (\ref{e118}) is a direct consequence of (\ref{e117}) in Proposition \ref{r1}.
\end{proof}

It follows from Theorem \ref{teo1} that, for each fixed $T >0$, the solution map
\begin{equation}\label{e100}
\A: L^2(\R) \rightarrow B_{0,T}, \quad \A u_0=u
\end{equation}
is well defined. Moreover, we have the following result:

\begin{prop}\label{prop2}
The solution map (\ref{e100}) is locally Lipschitz continuous, i.e, there exists a continuous function $C_0: \R^+\times (0,\infty) \rightarrow \R^+$, nondecreasing in its first variable, such that, for all $u_0, v_0 \in L^2(\R)$, we have
\begin{equation*}
\|\A u_0-\A v_0\|_{0,T} \leq C_0\left( \|u_0\|_2+\|v_0\|_2,T\right)\|u_0-v_0\|_2.
\end{equation*}
\end{prop}

\begin{proof}
Let $0 < \theta \leq T$ and $n=\left[ \frac{T}{\theta}\right]$. By Theorem  \ref{teo1},
\begin{equation}\label{e102}
\|\A u_0\|_{0,\theta}\leq 2C_{\theta} \|u_0\|_2,
\end{equation}
and
\begin{equation*}
\|\A u_0-\A v_0\|_{0, \theta} \leq C_{\theta} \left\lbrace \|u_0-v_0\|_2+\|M(\A u_0)-M(\A v_0)\|_{L^1(0,\theta;L^2(\R))}\right\rbrace,
\end{equation*}
where $C_{\theta}=2e^{\theta\|b\|_{\infty}}$. By Lemma \ref{lm2},
\begin{align*}
\|\A u_0-\A v_0\|_{0, \theta} &\leq C_{\theta} \|u_0-v_0\|_2  + C_{\theta}C\left\lbrace 2^{\frac{1}{2}}\theta^{\frac{1}{4}}\|\A u_0\|_{0,\theta} +2^{\frac{p}{2}}\theta^{\frac{2-p}{4}}\left(\|\A u_0\|_{0,\theta}^p+\|\A u_0\|_{0,\theta}\|\A v_0\|_{0,\theta}^{p-1} \right. \right. \notag\\
&\left.\left.+\|\A v_0\|_{0,\theta}^{p}\right) +\theta^{\frac{1}{2}}\right\rbrace\|\A u_0-\A v_0\|_{0,\theta},
\end{align*}
and applying (\ref{e102}), it follows that
\begin{align}
\|\A u_0-\A v_0\|_{0, \theta} &\leq C_{\theta} \|u_0-v_0\|_2  + C_{\theta}C\left\lbrace 2^{\frac{3}{2}}\theta^{\frac{1}{4}}C_{\theta}\|u_0\|_2 +2^{\frac{3p}{2}}\theta^{\frac{2-p}{4}}C_{\theta}^p\left(\|u_0\|_2^p+\|u_0\|_2\|v_0\|_2^{p-1} \right. \right. \notag\\
&\left.\left.+\|v_0\|_2^{p}\right) +\theta^{\frac{1}{2}}\right\rbrace\|\A u_0-\A v_0\|_{0,\theta} \notag \\
&\leq  C_{T} \|u_0-v_0\|_2  + C_{T}C \theta^{\frac{2-p}{4}}\left\lbrace 2^{\frac{3}{2}}\theta^{\frac{p-1}{4}}C_{T}(\|u_0\|_2+\|v_0\|_2) +2^{\frac{3p}{2}}C_{T}^p\left(\|u_0\|_2+\|v_0\|_2\right)^p \right. \notag\\
&\left. +\theta^{\frac{p}{4}}\right\rbrace\|\A u_0-\A v_0\|_{0,\theta} \notag \\
&\leq C_{T} \|u_0-v_0\|_2  + C_{T}C \theta^{\frac{2-p}{4}}\left\lbrace 2^{\frac{5}{2}}T^{\frac{p-1}{4}}C_{T}^2(\|u_0\|_2+\|v_0\|_2) +2^{\frac{5p}{2}}C_{T}^{2p}\left(\|u_0\|_2+\|v_0\|_2\right)^p \right. \notag\\
&\left. +T^{\frac{p}{4}}\right\rbrace\|\A u_0-\A v_0\|_{0,\theta}. \notag
\end{align}
Choosing  $\theta$ small enough, such that
\begin{equation}\label{e103}
\theta < \left[\frac{1}{2C_{T}C\left\lbrace 2^{\frac{5}{2}}T^{\frac{p-1}{4}}C_{T}\|u_0\|_2+\|v_0\|_2 +2^{\frac{5p}{2}}C_{T}^{2p} (\|u_0\|_2+\|v_0\|_2)^p+T^{\frac{p}{4}}\right\rbrace} \right]^{\frac{4}{2-p}},
\end{equation}
we have
\begin{equation}\label{e104}
\|\A u_0-\A v_0\|_{0, \theta} \leq 2C_{T} \|u_0-v_0\|_2.
\end{equation}
Analogously,  we can deduce that
\begin{equation*}
\|\A u_0\|_{0,[k\theta,(k+1)\theta]}\leq 2C_{\theta} \|u(k\theta)\|_2, \qquad k=0,1,...,n-1,
\end{equation*}
where $\|\cdot\|_{0,[k\theta,(k+1)\theta]}$ denotes the norm of
\begin{equation*}
B_{0,[k\theta,(k+1)\theta]}:=C\left([k\theta,(k+1)\theta];L^2(\R)\right) \cap L^2(k\theta,(k+1)\theta;H^1(\R)).
\end{equation*}
Moreover, by using the same arguments, we have
\begin{align*}
\|\A u_0-\A v_0\|_{0,[k\theta,(k+1)\theta]} &\leq C_{T} \|u(k\theta)-v(k\theta)\|_2  + C_{T}C \theta^{\frac{2-p}{4}}\left\lbrace 2^{\frac{3}{2}}T^{\frac{p-1}{4}}C_{T}(\|u(k\theta)\|_2+\|v(k\theta)\|_2) \right. \\
&\left.+2^{\frac{3p}{2}}C_{T}^p\left(\|u(k\theta)\|_2+\|v(k\theta)\|_2\right)^p +T^{\frac{p}{4}}\right\rbrace\|\A u_0-\A v_0\|_{0,[k\theta,(k+1)\theta]}.
\end{align*}
Combining (\ref{e102}) and the above estimate, it follows that
\begin{align*}
\|\A u_0-\A v_0\|_{0,[k\theta,(k+1)\theta]} &\leq C_{T} \|u(k\theta)-v(k\theta)\|_2  + C_{T}C \theta^{\frac{2-p}{4}}\left\lbrace 2^{\frac{5}{2}}T^{\frac{p-1}{4}}C_{T}^2(\|u_0\|_2+\|v_0\|_2) \right. \\
&\left.+2^{\frac{5p}{2}}C_{T}^{2p}\left(\|u_0\|_2+\|v_0\|_2\right)^p +T^{\frac{p}{4}}\right\rbrace\|\A u_0-\A v_0\|_{0,[k\theta,(k+1)\theta]}.
\end{align*}
Finally, from (\ref{e103}), we get
\begin{align}\label{e105}
\|\A u_0-\A v_0\|_{0,[k\theta,(k+1)\theta]} \leq 2C_{T} \|u(k\theta)-v(k\theta)\|_2, \quad k=0,1,...,n-1.
\end{align}
On the other hand, note that (\ref{e104}) and (\ref{e105}) imply that
\begin{align*}
\|\A u_0-\A v_0\|_{0,[k\theta,(k+1)\theta]} \leq 2^{k}C_{T}^{k} \|u_0-v_0\|_2, \quad k=0,1,...,n-1,
\end{align*}
and, therefore,
\begin{equation*}
\|\A u_0-\A v_0\|_{0,[k\theta,(k+1)\theta]} \leq 2^{n}C_{T}^{n} \|u_0-v_0\|_2.
\end{equation*}
Finally,
\begin{align*}
\|\A u_0-\A v_0\|_{0,T}&\leq \sum_{k=0}^{n-1} \|\A u_0-\A v_0\|_{0,[k\theta,(k+1)\theta]} \leq \sum_{k=0}^{n-1} 2^{n}C_{T}^{n} \|u_0-v_0\|_2  \\
&\leq  2^{n}C_{T}^{n}n \|u_0-v_0\|_2 \leq C_0(\|u_0\|_2+\|v_0\|_2)\|u_0-v_0\|_2,
\end{align*}
where $C_0(s)=\frac{T}{\theta(s)}\left[2C_T\right]^{\frac{T}{\theta(s)}}$.
\end{proof}

%---------------------------------------------------

Next, we will show well-posedness in $B_{3,T}$, with $1\leq p <2$. Therefore, let us first consider the following linearized problem given by
\begin{equation}\label{ee105}
\left\lbrace \begin{tabular}{l l}
$v_t +v_{xxx}-v_{xx}+[a(u)v]_x+bv=0$ & in $\R \times (0, \infty)$ \\
$v(0)=v_0$ & in $\R \times (0, \infty)$.
\end{tabular}\right.
\end{equation}
Then, we can establish the following proposition:

%%%%%%%%%%%%%%%%%%%%%%%%%%%%%%%%%%%%%%%%%%%%%%%%%%%%%%%%%%%%
\begin{prop}\label{prop4}
Let $a$ be a function $ C^1(\R)$ satisfying
\begin{gather*}\label{e107}
\text{$|a(\mu)| \leq C(1+|\mu|^{p})$ and $|a'(\mu)| \leq C(1+|\mu|^{p-1}), \quad \forall \mu \in \R$,}
\end{gather*}
with $1\leq p<2$. Let $T>0$, $b \in L^{\infty}(\R)$, $u \in B_{0,T}$ and $v_0 \in L^2(\R)$. Then, the problem (\ref{ee105}) admits a unique solution $v \in B_{0,T}$, such that
\begin{equation*}\label{e106}
\|v\|_{0,T} \leq \sigma(\|u\|_{0,T})\|v_0\|_2,
\end{equation*}
where $\sigma: \R^+ \rightarrow \R^+$ is a nondecreasing  continuous function.
\end{prop}

\begin{proof}
Let $0<\theta\leq T$ and $u \in B_{0,T}$. The proof of the existence follows the steps of  Proposition \ref{r1} and Theorem \ref{teo1}. Therefore, we will omit the details. First, Note that Lemma \ref{lm1} and Lemma \ref{lm2} imply that $Nw:=[a(u)w]_x \in L^{1}(0,\theta;L^2(\R))$, for all $w \in B_{0,\theta}$. Hence,
\begin{align*}\label{e108}
\|Nw\|_{L^1(0,\theta; L^2(\R))}&\leq C\left\lbrace 2^{\frac{1}{2}}\theta^{\frac{1}{4}}\|u\|_{0,\theta}\|w\|_{0,\theta} + 2^{\frac{p+2}{2}}\theta^{\frac{2-p}{4}}\|u\|_{0,\theta}^p\|w\|_{0,\theta}+\theta^{\frac{1}{2}}\|w\|_{0,\theta}\right\rbrace.
\end{align*}
With the notation above, problem (\ref{ee105}) takes the form
\begin{equation*}\label{e109}
\left\lbrace \begin{tabular}{l}
$v_t = A_bv - Nw$ \\
$v(0)=u_0$,
\end{tabular} \right.
\end{equation*}
where $A_bv=\partial_x^2v-\partial_x^3v-bv$. Since $A_b$ generates a strongly continuous semigroup $\{S(t)\}_{t\geq 0}$ of contractions in $L^2(\R)$, by Proposition \ref{prop5}, (\ref{e109}) has a unique mild solution $v \in B_{0,\theta}$, such that
\begin{equation*}\label{e110}
\|v\|_{0,\theta}\leq C_{\theta} \{ \|v_0\|_2+\|Nw\|_{L^1(0,\theta;L^2(\R))}\},
\end{equation*}
where $C_{\theta}=2e^{\theta\|b\|_{\infty}}$. Thus, we can define the operator
\[
\text{$\Gamma:B_{0,T}\longrightarrow B_{0,T}$ given by $\Gamma (w) = v$.}
\]
Let $R>0$ be a constant to be determined later and $w \in B_R(0):=\{w \in B_{0,\theta}: \|w\|_{B_{0,\theta}}\leq R\}$. Thus,
\begin{align*}
\|\Gamma w \|_{0,\theta} &\leq C_T \{\|v_0\|_2+  \left(2^{\frac{1}{2}}C\theta^{\frac{1}{4}}\|u\|_{0,T} + 2^{\frac{p+2}{2}}C\theta^{\frac{2-p}{4}}\|u\|_{0,T}^p+\theta^{\frac{1}{2}}C\right)R \}.
\end{align*}
By choosing $R=2C_T\|v_0\|_2$, we have
\begin{equation*}
\|\Gamma u \|_{0,\theta}\leq \left(K_1+\frac{1}{2}\right)R,
\end{equation*}
where $K_1= C_TC\left(2^{\frac{1}{2}}C\theta^{\frac{1}{4}}\|u\|_{0,T} + 2^{\frac{p+2}{2}}C\theta^{\frac{2-p}{4}}\|u\|_{0,T}^p+\theta^{\frac{1}{2}}\right).$ On the other hand, note that $\Gamma s - \Gamma w$ solves the problem
\begin{equation*}\label{e112}
\left\lbrace \begin{tabular}{l}
$v_t = A_bv - (Ns-Nw)$ \\
$v(0)=0$.
\end{tabular} \right.
\end{equation*}
Thus,
\begin{align*}
\|\Gamma s -\Gamma w \|_{0,\theta}\leq K_1\|s-w\|_{0,\theta}.
\end{align*}
Choosing $\theta>0$, such that $K_1=K_1(\theta)<\frac{1}{2}$, we have
\begin{equation*}
\left\lbrace\begin{tabular}{l}
$\|\Gamma w\|_{B_{0,\theta}}\leq R$ \\
$\|\Gamma s-\Gamma w\|_{B_{0,\theta}}< \frac{1}{2}\|s-w\|_{0,\theta}$
\end{tabular}\right., \qquad \forall s,w \in B_R(0) \subset B_{0,\theta}.
\end{equation*}
Hence, $\Gamma: B_R(0) \longrightarrow B_R(0)$ is a contraction and, by Banach fixed point theorem, we obtain a unique $v\in B_R(0)$, such that $\Gamma (v)=v$. Consequently, $v$ is a unique local mild solution of problem (\ref{ee105}) and
\begin{equation*}\label{e113}
\|v\|_{B_{0,\theta}}\leq 2C_T\|v_0\|_2.
\end{equation*}
Then, using standards arguments we may extend $\theta$ to $T$. Finally, the proof is completed defining $\sigma (s)=2C_T$.
\end{proof}

\noindent The aforementioned result is proved below. We make use of Proposition \ref{prop4} and classical energy-type estimates.

\begin{thm}\label{prop3}
Let $a$ be a function $ C^1(\R)$ satisfying
\begin{gather}\label{e66}
\text{$|a(\mu)| \leq C(1+|\mu|^{p})$ and $|a'(\mu)| \leq C(1+|\mu|^{p-1}), \quad \forall \mu \in  \R$,}
\end{gather}
with $1\leq p<2$. Let $T>0$, $b \in H^{1}(\R)$ and $u_0 \in H^3(\R)$. Then, there exists a unique mild solution $u \in B_{3,T}$ of $(\ref{e1})$, such that
\begin{equation*}\label{e68}
\|u\|_{3,T}\leq \beta_3(\|u_0\|_2)\|u_0\|_{H^3(\R)},
\end{equation*}
where $\beta_3: \R_+\rightarrow \R_+$ is a nondecreasing continuous function.
\end{thm}

\begin{proof}
In ordet to make the reading easier, the proof will be done in several steps:
\\
\vspace{1mm}

\noindent \textbf{Step 1: $u \in L^2(0,T; H^3(\R))$}

\vglue 0.1 cm

Since $u_0 \in H^3(\R) \hookrightarrow L^2(\R)$, by Theorem \ref{teo1}, there exist a unique solution $u \in B_{0,T}$, such that
\begin{equation}\label{e85}
\|u\|_{0,T} \leq \beta_0(\|u_0\|_2)\|u_0\|_2.
\end{equation}
We will show that $u \in B_{3,T}$. Let $v=u_t$. Then, $v$ solves the problem
\begin{equation*}\label{e67}
\left\lbrace \begin{tabular}{l}
$v_t+v_{xxx}-v_{xx}+[a(u)v]_x+bv=0$ \\
$v(0,x)=v_0$,
\end{tabular}\right.
\end{equation*}
where $v_0=-\partial_x^3u_0+\partial_x^2u_0-a(u_0)\partial_xu_0-bu_0$. Note that $v  \in L^2(\R)$ and there exist  $C=C(\|u_0\|_2)$, satisfying
\begin{equation*}\label{e69}
\|v_0\|_2 \leq C(\|u_0\|_2)\|u_0\|_{H^3(\R)}.
\end{equation*}
In fact, from (\ref{e6}) we can bound $v_0$ as follows:
\begin{align*}
\|v_0\|_2 &\leq \|\partial_x^3u_0\|_2+\|\partial_x^2u_0\|_2+\|a(u_0)\partial_xu_0\|_2+\|bu_0\|_2 \\
&\leq C_1 \left\lbrace (1+\|b\|_{L^{\infty}(\R)})\|u_0\|_{H^3(\R)}+\|u_0\|^{\frac{p}{2}}_{2}\|\partial_x u_0\|^{\frac{p+2}{2}}_2 \right\rbrace.
\end{align*}
Recall the Gagliardo-Nirenberg inequality:
\begin{equation}\label{e87}
\|\partial^j_xu_0\|_2 \leq C \|\partial^m_xu_0\|^{\frac{j}{m}}_2\|u_0\|_2^{1-\frac{j}{m}}, \quad j\leq m, \quad \text{where $j,m = 0,1,2,3$.}
\end{equation}
Applying (\ref{e87}) with $j=1$ and $m=2$, we have
\begin{align*}
\|v_0\|_2 &\leq C_2 \left\lbrace (1+\|b\|_{L^{\infty}(\R)})\|u_0\|_{H^3(\R)}+\|u_0\|^{\frac{3p+2}{4}}_{2}\|\partial_x^2 u_0\|^{\frac{p+2}{4}}_2\right\rbrace.
\end{align*}
Then, Young inequality guarantees that
\begin{align*}
\|v_0\|_2 &\leq C_3 \left\lbrace (1+\|b\|_{L^{\infty}(\R)})\|u_0\|_{H^3(\R)}+\|u_0\|^{\frac{4p}{2-p}}_{2}\|u_0\|_{2}+\|\partial_x^2 u_0\|_2\right\rbrace \\
\end{align*}
and, consequently, leader to
\begin{equation}\label{e127}
\|v_0\|_2\leq C(\|u_0\|_2)\|u_0\|_{H^3(\R)},
\end{equation}
where $C(s)=C_3 \left\lbrace 2+\|b\|_{L^{\infty}(\R)}+s^{\frac{4p}{2-p}}\right\rbrace$. Using Proposition \ref{prop4}, we see that $v \in B_{0,T}$ and
\begin{equation*}
\|v\|_{0,T} \leq  \sigma (\|u\|_{0,T})\|v_0\|_2.
\end{equation*}
Where $\sigma(s)=2C_T$. Combining (\ref{e85}) and (\ref{e127}), we get
\begin{equation}\label{e86}
\|v\|_{0,T} \leq \sigma (\beta_0(\|u_0\|_2)\|u_0\|_2)C(\|u_0\|_2)\|u_0\|_{H^3(\R)}.
\end{equation}
Then,
\begin{equation}\label{e75}
u, u_t \in L^2(0,T;H^1(\R))
\end{equation}
and, therefore,
\begin{equation}\label{e70}
u \in C([0,T];H^1(\R))  \hookrightarrow  C([0,T];C(\R)).
\end{equation}
On the other hand, note that $a(u)u_x, bu \in L^2(0,T;L^2(\R))$. In fact, from (\ref{e70}) it follows that
\begin{align*}
\|a(u)u_x\|^2_{L^2(0,T;L^2(\R))}&\leq C\left\lbrace\int^T_0\|u_x\|^2_2dx+\int^T_0\||u|^pu_x\|^2_2dx \right\rbrace \\
&\leq C\left\lbrace 1+\|u\|^{2p}_{C(0,T;C(\R))} \right\rbrace \|u\|_{0,T}^2
\end{align*}
and
\begin{align*}
\|bu\|_{L^2(0,T;L^2(\R))}\leq \|b\|_{L^{\infty}(\R)}\|u\|^2_{L^2(0,T;L^2(\R))}.
\end{align*}
Moreover, $u_{xxx}-u_{xx}=-u_t-a(u)u_x-bu$ in $D'(0,T,\R)$. Hence,
\begin{equation*}\label{e71}
u_{xxx}-u_{xx} = f \quad \in \quad L^2(0,T;L^2(\R), \quad \text{where $f:= -u_t-a(u)u_x-bu.$}
\end{equation*}
Taking Fourier transform, we have
\begin{equation}\label{e72}
\widehat{u}= \frac{\widehat{f}+\widehat{u}}{[1+\xi^2-i\xi^3]}
\end{equation}
and,
\begin{align}\label{e73}
\|u(t)\|^2_{H^3(\R)} \leq C_3 \left\lbrace \|f(t)\|_2^2+\|u(t)\|_2^2\right\rbrace
\end{align}
where $C_3=2\sup_{\xi \in \R}\dfrac{(1+\xi^2)^3}{(1+\xi^2)^2+\xi^6}$. Integrating (\ref{e73}) over $[0,T]$, we deduce that
\begin{equation}\label{e74}
u  \in L^2(0,T;H^3(\R)).
\end{equation}
\textbf{Step 2: $u \in B_{3,T}$}
\vglue 0.1 cm

First, observe that, according to (\ref{e75}), $u_t \in L^2(0,T;H^{-3}(\R))$. Then, considering the Hilbert triple $H^{3}(\R)\hookrightarrow H^{2}(\R)\hookrightarrow H^{-3}(\R)$,
by Lemma 1.2 in \cite[Chapter III]{temam1984theory}, we have
\begin{equation*}\label{e77}
u  \in C([0,T];H^2(\R)).
\end{equation*}
This implies further
\begin{equation}\label{e78}
u_{xx}, bu \in C([0,T];L^2(\R))\cap L^2(0,T;H^1(\R)).
\end{equation}
On the other hand, note that
\begin{align*}
\|a(u(t))u_x(t)-a(u(t_0))u_x(t_0)\|_{2} &\leq  \|[a(u(t))-a(u(t_0))]u_x(t)\|_{2}+\|a(u(t_0))[u_x(t)-u_x(t_0)]\|_{2} \\
&\leq C\left\lbrace \|(1+|u(t)|^{p-1}+|u(t_0)|^{p-1})|u(t)-u(t_0)|u_x(t)\|_2 \right.\\
&\left.+\|(1+|u(t_0)|^{p})|u_x(t)-u_x(t_0)|\|_2\right\rbrace\\
&\leq C\left\lbrace (1+\|u(t)\|^{p-1}_{\infty}+\|u(t_0)\|^{p-1}_{\infty})\|u(t)-u(t_0)\|_{\infty}\|u_x(t)\|_2 \right.\\
&\left.+(1+\|u(t_0)\|^{p}_{\infty})\|u_x(t)-u_x(t_0)\|_2\right\rbrace.\\
\end{align*}
Then, by (\ref{e70}) we have
\[
\lim_{t\rightarrow t_0}\|a(u(t))u_x(t)-a(u(t_0))u_x(t_0)\|_{2}=0
\]
and, therefore $a(u)u_x \in C([0,T];L^2(\R))$. The results above also guarantee that
\begin{equation}\label{e82}
a(u)u_x \in C([0,T];L^2(\R)) \cap  L^2(0,T;H^1(\R)).
\end{equation}
Indeed, it is sufficient to combine (\ref{e70}), (\ref{e74}) and the estimates
\begin{align*}
\|a'(u)u_x^2\|_{L^2(0,T;L^2(\R))}\leq C\left\lbrace (1+\|u\|_{C([0,T];C(\R))}^{p-1})\|u_x\|_{C([0,T];C(\R))}\|u_x\|_{L^2([0,T];L^2(\R))}\right\rbrace
\end{align*}
and
\begin{align*}
\|a(u)u_{xx}\|_{L^2(0,T;L^2(\R))}&\leq C \left\lbrace (1+\|u\|_{C([0,T];C(\R))}^{p})\|u_{xx}\|_{L^2([0,T];L^2(\R))}\right\rbrace.
\end{align*}
Since
\begin{align*}
u_{xxx}=-u_t+u_{xx}-a(u)u_x-bu,
\end{align*}
from (\ref{e75}), (\ref{e78}) and (\ref{e82}), we obtain
\begin{equation}\label{e83}
u \in  L^2(0,T;H^4(\R)).
\end{equation}
Finally, considering the Hilbert triple $H^{4}(\R)\hookrightarrow H^{3}(\R)\hookrightarrow H^{-4}(\R)$, %(Here $H^{3}(\R)$ is identified with its dual by the usual Riesz representation for Hilbert spaces),
Lemma 1.2  in \cite[Chapter III]{temam1984theory} gives
\begin{equation}\label{e84}
u \in  C([0,T];H^3(\R))
\end{equation}
and (\ref{e83})-(\ref{e84}) imply that $u \in B_{3,T}$.

\noindent \textbf{Step 3: $\|u\|_{C([0,T];H^3(\R))} \leq \sigma_1(\|u_0\|_2)\|u_0\|_{H^3(\R)}$}
\vglue 0.1 cm

First, note that, according to (\ref{e73}), the following estimate holds:
\begin{align}\label{e88}
\|u(t)\|_{H^3(\R)} \leq C_4\left\lbrace \|u_t(t)\|_2+\|a(u(t))u_x(t)\|_2+\|bu(t)\|_2+\|u(t)\|_2 \right\rbrace.
\end{align}
Next we combine (\ref{e66}), (\ref{e6})  and (\ref{e87}) with $j=1$ and $m=2$, to obtain
\begin{align*}
\|a(u(t))u_x(t)\|_2 &\leq C\left\lbrace\|u_x(t)\|_2+\|u(t)\|_2^{\frac{p}{2}}\|u_x(t)\|_2^{\frac{p+2}{2}}\right\rbrace \\
 &\leq C\left\lbrace \|u_{xx}(t)\|_2^{\frac{1}{2}}\|u(t)\|_2^{\frac{1}{2}}+\|u(t)\|_2^{\frac{3p+2}{4}}\|u_{xx}(t)\|_2^{\frac{p+2}{4}}\right\rbrace.
\end{align*}
Moreover,  Young inequality gives
\begin{align*}
\|a(u(t))u_x(t)\|_2 &\leq C_5\left(\|u(t)\|_2+\|u(t)\|_2^{\frac{3p+2}{2-p}}\right) +\frac{1}{2C_4}\|u(t)\|_{H^3(\R)}
\end{align*}
Replacing the estimate above in (\ref{e88}) and taking the supremum  in $[0,T]$, we get
\begin{align*}
\|u\|_{C([0,T];H^3(\R))} \leq 2C_4\left\lbrace \|u_t\|_{0,T}+(C_6+\|b\|_{\infty}) \|u\|_{0,T}+C_5\|u\|_{0,T}^{\frac{3p+2}{2-p}} \right\rbrace.
\end{align*}
Then, using (\ref{e85}) and (\ref{e86}) it follows that
\begin{align}\label{e89}
\|u\|_{C([0,T];H^3(\R))} &\leq 2C_4\left\lbrace \sigma (\beta_0(\|u_0\|_2)\|u_0\|_2)C(\|u_0\|_2)\|u_0\|_{H^3(\R)}+(C_6+\|b\|_{\infty})\beta_0(\|u_0\|_2)\|u_0\|_{2} \right. \notag \\
&\left.+C_5\beta_0^{\frac{3p+2}{2-p}}(\|u_0\|_2)\|u_0\|_{2}^{\frac{4p}{2-p}}\|u_0\|_{2} \right\rbrace \notag\\
&= \sigma_1(\|u_0\|_2)\|u_0\|_{H^3(\R)},
\end{align}
where $\sigma_1(s)=2C_4\left\lbrace \sigma (\beta_0(s)s)C(s)+(C_6+\|b\|_{\infty})\beta_0(s)+C_5\beta_0^{\frac{3p+2}{2-p}}(s)s^{\frac{4p}{2-p}}\right\rbrace.$
\\
\vspace{1mm}

\noindent \textbf{Step 4: $\|u_{xxxx}\|_{L^2(0,T;L^2(\R))} \leq \sigma_5(\|u_0\|_2)\|u_0\|_{H^3(\R)}$}
\vglue 0.1 cm

In order to conclude the proof it remains to prove that $u \in L^2(0,T;H^4(\R))$. To do that, we differentiate the equation with respect to $x$ to obtain
\begin{align}\label{e90'}
\|u_{xxxx}\|_{L^2(0;T,L^2(\R))} &\leq \|v\|_{0,T}+T^{\frac{1}{2}}\|u\|_{C(0;T,H^3(\R))}+\|a'(u)u_x^2\|_{L^2(0;T,L^2(\R))}+\|a(u)u_{xx}\|_{L^2(0;T,L^2(\R))} \notag  \\
&+\|[bu]_{x}\|_{L^2(0;T,L^2(\R))}.
\end{align}
The next steps are denoted to estimate the terms on the right side of (\ref{e90'}). First, observe that
\begin{align*}
\|[bu]_{x}\|_{L^2(0;T,L^2(\R))} &\leq \|b\|_{H^1(\R)}\|u\|_{L^2(0;T,H^{1}(\R))}+\|b\|_{H^{1}(\R)}\|u_{x}\|_{L^2(0;T,L^2(\R))} \\
&\leq 2\|b\|_{H^1(\R)}\|u\|_{0,T}.
\end{align*}
Then, from (\ref{e85}), (\ref{e86}) and (\ref{e89}), we obtain
\begin{equation}\label{e90}
\|u_{xxxx}\|_{L^2(0;T,L^2(\R))} \leq \sigma_2(\|u_0\|_2)\|u_0\|_{H^3(\R)}+\|a'(u)u_x^2\|_{L^2(0;T,L^2(\R))}+\|a(u)u_{xx}\|_{L^2(0;T,L^2(\R))}
\end{equation}
where $\sigma_2(s)=\sigma (\beta_0(s)s)C(s)+T^{\frac{1}{2}}\sigma_1(s)+2\|b\|_{H^1(\R)}\beta_0(s)$. Moreover, using (\ref{e6}) it follows that
\begin{align*}
\|a'(u(t))u_x^2(t)\|_2 &\leq C\left\lbrace \|u_x^2(t)\|_2+\||u(t)|^{p-1}u_x^2(t)\|_2 \right\rbrace \\
&\leq C_7\left\lbrace \|u_x(t)\|_2^{\frac{3}{2}}\|u_{xx}(t)\|_2^{\frac{1}{2}}+\|u(t)\|_{2}^{\frac{p-1}{2}}\|u_x(t)\|_2^{\frac{p+2}{2}}\|u_{xx}(t)\|_2^{\frac{1}{2}} \right\rbrace  \\
&\leq C_7\left\lbrace \|u_x(t)\|_2\|u(t)\|_{H^3(\R)}+\|u(t)\|_{2}^{\frac{p-1}{2}}\|u_x(t)\|_2^{\frac{p+2}{2}}\|u_{xx}(t)\|_2^{\frac{1}{2}} \right\rbrace.
\end{align*}
Then,  Gagliardo-Nirenberg inequality (\ref{e87}) with $j=1$ and $m=3$ leads to
\begin{align*}
\|a'(u(t))u_x^2(t)\|_2 &\leq C_8\left\lbrace \|u_x(t)\|_2\|u(t)\|_{H^3(\R)}+\|u(t)\|_{2}^{\frac{5p+2}{6}}\|u_{xxx}(t)\|_2^{\frac{p+4}{6}}  \right\rbrace.
\end{align*}
Moreover, Young inequality gives
\begin{align*}
\|a'(u(t))u_x^2(t)\|_2^2  \leq  C_9\left\lbrace \|u_x(t)\|_2\|u(t)\|_{H^3(\R)}+\|u(t)\|_{2}^{\frac{5p+2}{2-p}}+\|u_{xxx}(t)\|_2  \right\rbrace,
\end{align*}
which allows us to conclude that
\begin{align*}
\|a'(u)u_x^2\|_{L^2(0,T;L^2(\R))} \leq C_{10}\left\lbrace \|u\|_{C([0,T];H^3(\R))} \|u\|_{0,T}+T^{\frac{1}{2}}\|u\|_{0,T}^{\frac{(5p+2)}{2-p}}+T^{\frac{1}{2}}\|u\|_{C([0,T];H^3(\R))}  \right\rbrace.
\end{align*}
Hence
\begin{equation}\label{e91}
\|a'(u)u_x^2\|_{L^2(0,T;L^2(\R))} \leq \sigma_3(\|u_0\|_2)\|u_0\|_{H^3(\R)}
\end{equation}
whit $\sigma_3(s)=C_{10}\left\lbrace\sigma_1(s)\beta_0(s)s+T^{\frac{1}{2}}\beta_0^{\frac{(5p+2)}{2-p}}(s)s^{\frac{(5p+2)}{2-p}-1}+T^{\frac{1}{2}}\sigma_1(s)\right\rbrace$. On the other hand, (\ref{e6}) yields
\begin{align*}
\|a(u(t))u_{xx}(t)\|_2  &\leq C_{11} \left\lbrace \|u(t)\|_{H^3(\R)}+\|u(t)\|_2^{\frac{p}{2}}\|u_{x}(t)\|_2^{\frac{p}{2}}\|u_{xx}(t)\|_2   \right\rbrace \\
&\leq C_{11} \left\lbrace \|u\|_{C([0,T];H^3(\R))}+\|u\|_{0,T}^{\frac{p}{2}}\|u\|_{C([0,T];H^3(\R))}\|u_{x}(t)\|_2^{\frac{p}{2}}   \right\rbrace.
\end{align*}
It therefore transpire that
\begin{align*}
\|a(u)u_{xx}\|_{L^2(0,T;L^2(\R))}  &\leq C_{12} \left\lbrace T^{\frac{1}{2}} \|u\|_{C([0,T];H^3(\R))}+\|u\|_{0,T}^{\frac{p}{2}}\|u\|_{C([0,T];H^3(\R))}\left(\int_0^T\|u_{x}(t)\|_2^p\right)^{\frac{1}{2}}   \right\rbrace \\
&\leq C_{12} \left\lbrace T^{\frac{1}{2}} \|u\|_{C([0,T];H^3(\R))}+T^{\frac{2-p}{4}}  \|u\|_{0,T}^{p}\|u\|_{C([0,T];H^3(\R))}\right\rbrace.
\end{align*}
from which one obtains the inequality
\begin{align}\label{e92}
\|a(u)u_{xx}\|_{L^2(0,T;L^2(\R))} &\leq \sigma_4(\|u_0\|_2)\|u_0\|_{H^3(\R)},
\end{align}
whit $\sigma_4(s)=C_{12}\left\lbrace T^{\frac{1}{2}}\sigma_1(s)+T^{\frac{2-p}{4}}\beta_0^p(s)\sigma_1(s)s^p \right\rbrace$. Consequently, (\ref{e90}), (\ref{e91}) and (\ref{e92}) lead to
\begin{equation}\label{e93}
\|u_{xxxx}\|_{L^2(0;T,L^2(\R))} \leq \sigma_5(\|u_0\|_2)\|u_0\|_{H^3(\R)},
\end{equation}
where $\sigma_5(s)=\sigma_2(s)+\sigma_3(s)+\sigma_4(s)$. Finally, using (\ref{e89}) and (\ref{e93}), we conclude that $u \in L^(0,T;H^4(\R))$ and
\begin{align*}
\|u\|_{3,T} \leq \beta_3(\|u_0\|_2)\|u_0\|_{H^3(\R)},
\end{align*}
where  $\beta_3(s)=\sigma_1(s)+\sigma_5(s)$.
\end{proof}

Next, we will show the well-posedness of the IVP (\ref{e1}) in the space $H^s(\R)$ for $0\leq s \leq 3$ and $1\leq p <2$. In order to do that, we will used a method introduced by Tartar \cite{tartar1972interpolation} and adapted by Bona and Scott \cite[Theorem 4.3]{bona1976solutions} to prove the global well-posedness of the pure initial value problem for the KdV equation on the whole line in fractional order Sobolev spaces $H^s(\R)$.

Let $B_0$ and $B_1$  be two Banach spaces, such that $B_1\subset B_0$ with the inclusion map being continuous. Let $f \in B_0$, for $t \geq 0$ and define
\begin{align*}
K(f,t)=\inf_{g \in B_1}\left\lbrace \|f-g\|_{B_0}+t\|g\|_{B_1}\right\rbrace.
\end{align*}
For $0<\theta <1$ and $1\leq p \leq +\infty$, define
\begin{align*}
\B_{\theta,p}:=[B_0,B_1]_{\theta,p}=\left\lbrace f\in B_0: \|f\|_{\theta,p}:\left(\int_{0}^{\infty}K(f,t)t^{-\theta p - 1}dt\right)^{\frac{1}{p}}< \infty \right\rbrace
\end{align*}
with the usual modification for the case $p=\infty$. Then, $B_{\theta,p}$ is a Banach space with norm $\|\cdot\|_{\theta,p}$. Given two pairs $(\theta_1,p_1)$ and $(\theta_2,p_2)$ as above,  $(\theta_1,p_1)\prec(\theta_2,p_2)$ will denote
\begin{equation*}
\left\lbrace \begin{tabular}{l l}
$\theta_1 < \theta_2$ & or \\
$\theta_1 = \theta_3$ &  and $p_1>p_2$.
\end{tabular}\right.
\end{equation*}
If $(\theta_1,p_1)\prec(\theta_2,p_2)$, then $\B_{\theta_2,p_2}\subset \B_{\theta_1,p_1}$ with the inclusion map continuous.

Then, the following result holds:

\begin{thm}\label{inter}
Let $B_{0}^j$ and $B_{1}^j$ be Banach spaces such that $B_{1}^j\subset B_{0}^j$ with continuous inclusion mappings, for $j=1,2$. Let $\alpha$ and $q$ lie in the ranges $0 < \alpha < 1$ and $1\leq q\leq \infty$. Suppose that $\A$ is a mapping satisfying
\begin{enumerate}
\item[(i)] $\A: \B_{\alpha, q}^{1} \rightarrow B_0^2$ and, for $f,g \in  \B_{\alpha, q}^{1}$,
\begin{equation*}
\|\A f - \A g\|_{B_0^2}\leq C_0\left(\|f\|_{\B_{\alpha, q}^{1}}+\|g\|_{\B_{\alpha, q}^{1}} \right)\|f-g\|_{B_0^1},
\end{equation*}
\item[(ii)] $\A: B_1^1 \rightarrow B_1^2$ and, for $h\in  \B_1^{1}$,
\begin{equation*}
\|\A h\|_{B_1^2}\leq C_1\left(\|h\|_{\B_{\alpha, q}^{1}} \right)\|h\|_{B_1^1},
\end{equation*}
\end{enumerate}
where $C_j : \R^+ \rightarrow \R^+$ are continuous nondecreasing functions, for $j=0,1$. Then, if $(\theta,p)\geq(\alpha,q)$, $\A$ maps $\B_{\theta, p}^{1}$ into $\B_{\theta, p}^{2}$ and, for $f \in \B_{\theta, p}^{1}$, we have
\begin{equation*}
\|\A f\|_{\B_{\theta, p}^2}\leq C\left(\|f\|_{\B_{\alpha, q}^{1}} \right)\|f\|_{\B_{\theta,p}^1},
\end{equation*}
where $C(r)=4C_0(4r)^{1-\theta}C_1(3r)^{\theta}$, $r>0$.
\end{thm}

\begin{proof}
See \cite[Theorem 4.3]{bona1976solutions}
\end{proof}

This theorem leads to the main result of this section.

\begin{thm}\label{teo2}
Let $a$ be a $C^1(\R)$ function satisfying
\begin{equation*}
|a(\mu)|\leq C(1+|\mu|^p), \quad |a'(\mu)|\leq C(1+|\mu|^{p-1}), \qquad \forall \mu \in \R,
\end{equation*}
with $1 \leq p < 2$, and let $T>0$ and $0 \leq s\leq 3$ be given. In addition, assume that $b \in L^{\infty}(\R)$ when $s=0$ and $b\in H^1(\R)$ when $s>0$. Then, for any $u_0 \in H^s(\R)$, the IVP (\ref{e1}) admits a unique solution $u \in B_{s,T}$. Moreover, there exists a nondecreasing continuous function $\beta_s: \R^+ \rightarrow \R^+$, such that
\begin{equation*}
\|u\|_{B_{s,T}} \leq \beta_s(\|u_0\|_2)\|u_0\|_{H^s(\R)}.
\end{equation*}
\end{thm}

\begin{proof}
We define
\begin{equation*}
B_0^1=L^2(\R),\quad B_0^2=B_{0,T},\quad B_1^1=H^3(\R)\quad \text{and}\quad  B_1^2=B_{3,T}.
\end{equation*}
Thus,
\begin{equation*}
\B_{\frac{s}{3},2}^1=[L^2(\R),H^3(\R)]_{\frac{s}{3},2}=H^s(\R) \quad \text{and}\quad  \B_{\frac{s}{3},2}^2=[B_{0,T},B_{3,T}]_{\frac{s}{3},2}=B_{s,T}.
\end{equation*}
Combining Proposition \ref{prop2} and Theorem \ref{prop3} we obtain $(i)$ and $(ii)$ in the Theorem \ref{inter}. Then, Theorem \ref{inter} yields the result.
\end{proof}

Theorem \ref{teo2} gives a strong smoothing property for the solutions of the problem.

\begin{cor}\label{cor1}
Under the assumptions of Theorem \ref{teo2}, for any $u_0 \in L^2(\R)$, the corresponding solution $u$ of (\ref{e1}) belongs to
\begin{equation*}
B_{3,[\varepsilon,T]}=C([\varepsilon, T];H^3(\R))\cap L^2(\varepsilon, T;H^4(\R)),
\end{equation*}
for every $T>0$ and $0 <\varepsilon<T$.
\end{cor}

\begin{proof} The same result was obtained for the generalized KdV and the KdV-Burgers equations in \cite{rosier2006global} and \cite{cavalcanti2014global}, respectively.
Since the proof is analogous and follows from classical arguments we omit it.
\end{proof}

\subsection{\texorpdfstring{Case $p\geq 2$.}{Case 2}}

We first restrict ourselves to case $2\leq p < 5$ to obtain the existence of solutions in the $L^2$-setting, i. e., finite energy solutions. Next we prove the global well-posedness in the space $B_{3,T}$.

First, the reader is reminded of the following result which follows from the Egoroff theorem.

\begin{lem}\label{lm5}
Let $X$ be measure space and ${f_n}\in L^p(X)$, with $1\leq p < \infty$, such that $f_n \rightharpoonup  f$ in $L^p(X)$ and $f_n(x) \rightarrow g(x)$ a.e in $X$. Then $f(x)=g(x)$ a.e in $X$
\end{lem}

We also note that, different from the case $1\leq p <2$, the next result is not obtained combining semigroup theory and fixed point arguments. Here, due to some
technical problems, the solution is obtained as limit of the regulars ones. We follows the ideas contained in \cite{rosier2006global}.

\begin{thm}\label{teo4}
Let $a$ be a $C^1(\R)$ function satisfying
\begin{equation*}\label{e35}
|a(\mu)| \leq C(1+|\mu|^{p}) \qquad  |a'(\mu)| \leq C(1+|\mu|^{p-1}), \qquad \forall \mu \in R,
\end{equation*}
with $2 \leq p < 5$. Then, for any $u_0 \in L^2(\R)$ the problem $(\ref{e1})$ admits at least one solution $u$, such that
\[
u \in C_w([0,T];L^2(\R))\cap L^2(0,T;H^1(\R)),\quad \text{for all $T>0$}.
\]
\end{thm}

\begin{proof}
Consider a sequence $\{a_n\} \in C^{\infty}_0(\R),$ such that
\begin{align}\label{e41}
|a_n^{(j)}(\mu)|\leq C(1+|\mu|^{p-j}), \quad \forall \mu \in \R, \quad j=0,1,
\end{align}
\begin{equation}\label{e42}
a_n \longrightarrow a \quad \text{uniformly in each compact set in $\R$}.
\end{equation}
Note that $|a_n(\mu)|\leq C_n(1+|\mu|)$ and $|a'_n(\mu)|\leq C_n$. Then, for each $n$, Theorem \ref{teo1} guarantees the existence of a function $u_n \in B_{0,T}$ solution of
\begin{equation}\label{e43}
\left\lbrace \begin{tabular}{l}
$\partial_t u_n +\partial_x^3u_n-\partial_x^2u_n+a_n(u_n)\partial_xu_n+b(x)u_n=0$  \\
$u_n(0,x)=u_0(x),$
\end{tabular}\right.
\end{equation}
with $\|u_n\|_{0,T}\leq 2C_T\|u_0\|_{L^2(\R)}$. Hence,
\begin{equation}\label{e48}
\{u_n\} \quad \text{ is bounded in $C([0,T];L^2(\R))\cap L^2(0,T;H^1(\R))$. }
\end{equation}
From (\ref{e48}) we obtain a function $u$ and a subsequence, still denoted by the same index $n$, such that
\begin{equation}\label{e44}
u_n \rightharpoonup u \quad \text{in $L^{\infty}(0,T;L^2(\R))$ weak $*$}
\end{equation}
\begin{equation}\label{e45}
u_n \rightharpoonup u \quad \text{in $L^{2}(0,T;H^1(\R))$ weak.}
\end{equation}
In order to analyze the nonlinear term $a_n(u_n)\partial_xu_n$ we consider the functions
\begin{equation}\label{e46}
A(u):=\int_0^ua(v)dv \quad \text{and} \quad A_n(u):=\int_0^ua_n(v)dv.
\end{equation}
Note that $a_n(u_n)\partial_xu_n = \partial_x [A_n(u_n)]$. Then, taking $\alpha \in \left(1, \frac{6}{p+1}\right)$ and proceeding as in \cite[proof of Theorem 2,14]{rosier2006global}), we deduces that for each bounded set $I \subset \R$, the sequence $\{A_n(u_n)\}$ is bounded in $L^{\alpha}([0,T]\times I)$. Indeed,
\[
|A_n(u)|^{\alpha}\leq C\left( 2|u|+\frac{|u|^{p+1}}{p+1}\right)^{\alpha} \leq C'\left( |u|^{\alpha}+|u|^{\alpha(p+1)}\right)
\]
for some $C, C' >0$. Then,
\begin{align*}
\|A_n(u_n)\|^{\alpha}_{L^{\alpha}((0,T)\times I)} &\leq C''\left( \|u_n\|_{L^{2}(0,T;L^{2}(I))}^{\alpha} + \int_0^T\|u_n(t)\|^{\alpha(p+1)-2}_{\infty}\|u_n(t)\|_2 ^2dt \right) \\
&\leq C''\left( \|u_n\|_{0,T}^{\alpha}+ 2^{\frac{\alpha(p+1)-2}{2}}\int_0^T\|u_n(t)\|^{\frac{\alpha(p+1)+2}{2}}_2\|u_{nx}(t)\|^{\frac{\alpha(p+1)-2}{2}}_2dt \right) \\
&\leq C''\left(\|u_n\|_{0,T}^{\alpha} + 2^{\frac{\alpha(p+1)-2}{2}}T^{\frac{6-\alpha(p+1)}{4}}\|u_n\|^{\frac{\alpha(p+1)+2}{2}}_{0,T}\|u_n\|^{\frac{\alpha(p+1)-2}{2}}_{0,T}\right) \\
&\leq C''\left( \|u_n\|_{0,T}^{\alpha}+\|u_n\|^{\alpha(p+1)}_{0,T}\right)\\
&\leq \widetilde{C}\left( \|u_0\|_{2}^{\alpha}+\|u_0\|^{\alpha(p+1)}_{2}\right)
\end{align*}
Where $\widetilde{C}$ is a positive constant. Consequently,
\begin{equation*}\label{e61}
\{A_n(u_n)\} \quad \text{is bounded in $L^{\alpha}(0,T;H^{-1}(I))$} \quad (\text{since $L^{\alpha}(I) \hookrightarrow  H^{-1}(I)$})
\end{equation*}
and
\begin{equation}\label{e65}
\{a_n(u_n)\partial_xu_n\}=\{\partial_x[A_n(u_n)]\} \quad \text{is bounded in $L^{\alpha}(0,T;H^{-2}(I))$}.
\end{equation}
Moreover,  (\ref{e48}) and the fact that $1< \alpha \leq 2$ allow us to conclude that
\begin{equation*}
\text{$\{\partial^3_xu_n\}, \{\partial^2_xu_n\}$ and $\{bu_n\}$
are bounded in  $L^{2}(0,T;H^{-2}(\R)) \subset L^{\alpha}(0,T;H^{-2}(\R))$}
\end{equation*}
and, therefore,
\begin{equation}\label{e47}
\text{$\partial_tu_n=-\partial_x^3u_n+\partial_x^2u_n-a_n(u_n)\partial_xu_n-bu_n$ is bounded in $L^{\alpha}(0,T;H^{-2}(I)) \subset L^{1}(0,T;H^{-2}(\R))$}.
\end{equation}
Since $\{u_n\}$ is bounded in  $L^{\alpha}(0,T;H^{1}(\R))$ and the first embedding in $H^1(I)\hookrightarrow L^2(I) \hookrightarrow H^{-2}(I)$ is compact, we can apply Corollary 33 in \cite{simon1986compact} to conclude that $\{u_n\}$ is relative compact in $L^{2}(0,T;L^2(I))$. Consequently, there exists a subsequence, still denoted by  $\{u_n\}$, such that
\begin{equation}\label{e50}
\text{$u_{n} \longrightarrow u$ in $L^{2}(0,T;L^2_{loc}(\R))$ strongly and a.e.}
\end{equation}
Moreover, by (\ref{e45}),
\[
u_n \rightharpoonup u \quad \text{weak in $L^{2}(0,T;L^2(\R))\equiv L^{2}(\R\times (0,T))$}
\]
and by applying Lemma \ref{lm5}, we obtain
\begin{equation}\label{e51}
\text{$u_n \longrightarrow u$  a.e in $\R \times (0,T).$}
\end{equation}
Then, using (\ref{e42}), (\ref{e46}) and (\ref{e51}), it is easy to see that
\begin{equation*}\label{e52}
\text{$A_n(u_n(x,t)) \longrightarrow A(u(x,t))$ a.e in $\R \times (0,T).$}
\end{equation*}
Next, proceeding as in the previous steps and by applying Lemma \ref{lm5}, the following convergence holds
\begin{equation*}\label{e54}
\text{$A_{n}(u_{n}) \rightharpoonup A(u)$ weak in $L^{\alpha}(0,T; L^{\alpha}_{loc}(\R)).$}
\end{equation*}
Therefore, $A_{n}(u_{n}) \longrightarrow A(u)$ in $D'(\R\times (0,T))$ and, by taking the partial derivative, we obtain
\begin{equation}\label{e55}
\text{$a_n(u_n)\partial_xu_n \longrightarrow a(u)\partial_x u$ in $D'(\R\times(0,T)).$}
\end{equation}
From (\ref{e50}) and (\ref{e55}), we can take the limit in (\ref{e43}) to conclude that $u$ solves the equation (\ref{e1}) in the sense of distribution, i.e,
\begin{equation}\label{e56}
u_t+u_{xxx}-u_{xx}+a(u)u_x+bu=0 \quad \text{in $D'(\R \times (0,T))$.}
\end{equation}
Now, note that (\ref{e45}) yields
\begin{align*}
&u_{xxx} \in L^2(0,T;H^{-2}(\R))\hookrightarrow  L^{\alpha}(0,T;H^{-2}(\R)),  \\
&u_{xx} \in L^2(0,T;H^{-1}(\R))\hookrightarrow  L^{\alpha}(0,T;H^{-2}(\R)),  \\
&bu \in L^2(0,T;H^{1}(\R))\hookrightarrow  L^{\alpha}(0,T;H^{-2}(\R)).
\end{align*}
Moreover, by (\ref{e65}) $a(u)u_x=[A(u)u]_x \in L^{\alpha}(0,T;H^{-2}_{loc}(\R))$. Then, from (\ref{e56}), we deduce that $u_t \in L^{\alpha}(0,T;H^{-2}_{loc}(\R))$ and by (\ref{e45}),  %$u\in L^2(0,T;H^{-2}(\R))$;
we can conclude that $u \in C([0,T];H^{-2}_{loc}(\R))$. On the other hand, by (\ref{e48}) and (\ref{e47}), we infer from Corollary 33 in \cite{simon1986compact} that $\{u_n\}$ is relatively compact in $C([0,T];H^{-1}_{loc}(\R))$. Therefore, there exists and subsequence (denoted by $\{u_n\}$), such that
\begin{equation}\label{e58}
\text{$u_n \longrightarrow u$ in $C([0,T];H^{-1}_{loc}(\R))$.}
\end{equation}
In particular, $u(x,0)=\lim_{n\rightarrow \infty}u_n(x,0)=u_0(x)$. Finally, note that
\[
u \in L^{\infty}(0,T;L^2(\R))\cap C_w([0,T];H^{-1}_{loc}(\R))
\]
and from Lemma 1.4 in \cite[Ch III]{temam1984theory}, it follows that $u\in C_w([0,T];L^{2}(\R)).$

\end{proof}

\begin{mydef}\label{def2}
Let $T>0$, A function $u \in C_w([0,T];L^2(\R))\cap L^2(0,T;H^1(\R))$ is said to be a weak solution of problem (\ref{e1}) if there exist a sequence $\{a_n\}$ of function in $C_0^{\infty}(\R)$ satisfying (\ref{e41}) and (\ref{e42}) and a sequence of strong solution $u_n$ to (\ref{e43}), such that (\ref{e44}),(\ref{e45}),(\ref{e51}) and (\ref{e58}) hold true.
\end{mydef}

The proof of next result also requires an adaptation of Lemma \ref{lm2} as follows.

\begin{lem}\label{lm3}
For any $T>0$, $ p \geq 1$ and $u, v, w \in B_{3,T}$, such that $u_t, v_t, w_t \in B_{0,T}$,  we have
\begin{enumerate}
\item[(i)] $\|(a(u)v_x)_x\|_{L^2(0,T;L^2(\R))} \leq CT^{\frac{1}{2}} \left\lbrace \|u\|_{3,T}\|v\|_{3,T}+2\|u\|_{3,T}^p\|v\|_{3,T}  +\|v\|_{3,T}\right\rbrace$,
\item[(ii)]
\begin{align*}
\|a(u)v_x\|_{W^{1,1}(0,T;L^2(\R))} \leq CT^{\frac{1}{2}}& \left\lbrace (\|v\|_{3,T}+\|v_t\|_{0,T})+ \|u\|_{3,T}^p(\|v\|_{3,T}+\|v_t\|_{0,T})+ \|u_t\|_{0,T}\|v\|_{3,T}\right. \\
&\left.+\|u_t\|_{0,T}\|v\|_{3,T}\|u\|_{3,T}^{p-1}\right\rbrace,
\end{align*}
\item[(iii)] $\|uw_x\|_{W^{1,1}(0,T;L^2(\R))} \leq 2^{\frac{1}{2}}T^{\frac{1}{4}}\left\lbrace \|u\|_{3,T}\|w\|_{3,T}+ \|u_t\|_{0,T}\|w\|_{3,T}+\|u\|_{3,T}\|w_t\|_{0,T}\right\rbrace$.
\item[(iv)] If $p \geq 2$, then
\begin{align*}
\|u|w|^{p-1}v_x\|_{W^{1,1}(0,T;L^2(\R))} \leq  &T^{\frac{1}{2}}\left\lbrace \|u\|_{3,T}\|w\|_{3,T}^{p-1}(\|v\|_{3,T}+\|v_t\|_{0,T})+\|v\|_{3,T}\|w\|_{3,T}^{p-1}\|u_t\|_{0,T} \right.\\
&\left.+(p-1)\|u\|_{3,T}\|w\|_{3,T}^{p-2}\|v\|_{3,T}\|w_t\|_{0,T}\right\rbrace.
\end{align*}
\end{enumerate}
\end{lem}

\begin{proof}
First, note that if $u\in B_{3,T}$ we have
\begin{equation}\label{e25}
\begin{tabular}{l}
$\left\lbrace\begin{tabular}{l}
$\partial^j_x u \in C([0,T];H^{3-j}(\R))\hookrightarrow C([0,T];C(\R)) $ \\
$\|\partial^j_x u\|_{C([0,T];C(\R))}\leq C\|u\|_{3,T}$
\end{tabular}\right., \quad j=0,1,2,$ \\
\\
$\left\lbrace\begin{tabular}{l}
$\partial^j_x u \in L^2([0,T];H^{4-j}(\R))\hookrightarrow L^2([0,T];L^2(\R))$\\
$\|\partial^j_x u\|_{L^2([0,T];L^2(\R))}\leq C\|u\|_{3,T}$
\end{tabular}\right., \quad j=0,1,2,3.$
\end{tabular}
\end{equation}

\begin{enumerate}
\item[(i)] (\ref{e2}) and (\ref{e25}) imply that
\begin{align*}
\|(a(u)v_x)_x\|_{L^2(0,T;L^2(\R))} &\leq C\{ T^{\frac{1}{2}}\|u_x\|_{C([0,T];C(\R))}\|v_x\|_{C(0,T;L^2(\R))} \\
&+ T^{\frac{1}{2}}\|u\|_{C([0,T];C(\R))}^{p-1}\|u_x\|_{C([0,T];C(\R))}\|v_x\|_{C(0,T;L^2(\R))}+T^{\frac{1}{2}}\|v_{xx}\|_{C([0,T];L^2(\R))} \\
&+ T^{\frac{1}{2}}\|u\|_{C([0,T];C(\R))}^{p}\|v_{xx}\|_{C(0,T;L^2(\R))} \} \\
&\leq CT^{\frac{1}{2}} \left\lbrace \|u\|_{3,T}\|v\|_{3,T}+2\|u\|_{3,T}^p\|v\|_{3,T}  +\|v\|_{3,T}\right\rbrace.
\end{align*}
\item[(ii)] By (\ref{e2}), H\"{o}lder inequality and (\ref{e25}) we get
\begin{align*}
\|a(u)v_x\|_{W^{1,1}(0,T;L^2(\R))} &\leq C \left\lbrace \|v_x\|_{L^2(0,T;L^2(\R))}T^{\frac{1}{2}} +\|u\|_{C([0,T];C(\R))}^p \|v_x\|_{L^2(0,T;L^2(\R))}T^{\frac{1}{2}} \right.\\
&+ \|v_x\|_{C([0,T];C(\R))} \|u_t\|_{L^2(0,T;L^2(\R))}T^{\frac{1}{2}}  \\
&+ \|u\|_{C([0,T];C(\R))}^{p-1} \|v_x\|_{C([0,T];C(\R))}\|u_t\|_{L^2(0,T;L^2(\R))}T^{\frac{1}{2}} \\
&\left. +T^{\frac{1}{2}} \|v_{tx}\|_{L^2(0,T;L^2(\R))}+\|u\|_{C([0,T];C(\R))}^p\|v_{tx}\|_{L^2(0,T;L^2(\R))}T^{\frac{1}{2}} \right\rbrace \\
&\leq C T^{\frac{1}{2}} \left\lbrace \|v\|_{3,T}+\|u\|_{3,T}^p\|v\|_{3,T}+\|v\|_{3,T}\|u_t\|_{0,T}+\|u\|_{3,T}^{p-1}\|v\|_{3,T}\|u_t\|_{0,T} \right.\\
&\left. +\|v_t\|_{0,T}+\|u\|_{3,T}^p\|v_t\|_{0,T}\right\rbrace.
 \end{align*}
\item[(iii)] It is a consequence of (ii) in Lemma \ref{lm2}.
\item[(iv)] H\"{o}lder inequality and (\ref{e25}) lead to the desired result:
\begin{align*}
\|u|w|^{p-1}v_x\|_{W^{1,1}(0,T;L^2(\R))} &\leq \|u\|_{C([0,T];C(\R))} \|w\|_{C([0,T];C(\R))}^{p-1} \|v_x\|_{L^2(0,T;L^2(\R))}T^{\frac{1}{2}} \\
&+\|v_x\|_{C([0,T];C(\R))} \|w\|_{C([0,T];C(\R))}^{p-1} \|u_t\|_{L^2(0,T;L^2(\R))}T^{\frac{1}{2}}  \\
&+(p-1)\|u\|_{C([0,T];C(\R))} \|w\|_{C([0,T];C(\R))}^{p-2} \|v_x\|_{C([0,T];C(\R))}\|w_t\|_{L^2(0,T;L^2(\R))}T^{\frac{1}{2}} \\
&+ \|u\|_{C([0,T];C(\R))} \|w\|_{C([0,T];C(\R))}^{p-1} \|v_{tx}\|_{L^2(0,T;L^2(\R))}T^{\frac{1}{2}} \\
&\leq T^{\frac{1}{2}}\left\lbrace \|u\|_{3,T}\|w\|_{3,T}^{p-1}\|v\|_{3,T}+\|v\|_{3,T}\|w\|_{3,T}^{p-1}\|u_t\|_{0,T} \right.\\
&\left.+(p-1)\|u\|_{3,T}\|w\|_{3,T}^{p-2}\|v\|_{3,T}\|w_t\|_{0,T}+\|u\|_{3,T}\|w\|_{3,T}^{p-1}\|v_t\|_{0,T}\right\rbrace.
\end{align*}
\end{enumerate}

\end{proof}

Proposition \ref{prop5} asserts that the inhomogeneous linear problem (\ref{ee1}) is well-posedness and we have the existence of a mild solution. However, we can have a regular solution as shows the following result.

\begin{prop}\label{prop8}
Let $T>0$, $b \in H^1(\R)$ and $u_0 \in H^3(\R)$. If $f \in W^{1,1}(0,T;L^2(\R))$ and $f_x \in L^2(0,T;L^2(\R))$, the inhomogeneous linear problem (\ref{ee1}) has a unique regular solution $u \in B_{3,T}$, such that
\begin{equation}\label{ecc1}
\|u\|_{3,T}\leq C_{3,T}\left\lbrace\|u_0\|_{H^3(\R)}+\|f\|_{W^{1,1}(0,T;L^2(\R))}+\|f_x\|_{L^2(0,T;L^2(\R))} \right\rbrace,
\end{equation}
$u_t \in B_{0,T}$ and
\begin{equation}\label{ecc2}
\|u_t\|_{0,T}\leq C_{0,T}\left\lbrace\|u_0\|_{H^3(\R)}+\|f(0)\|_{L^2(\R)}+\|f_t\|_{L^1(0,T;L^2(\R))} \right\rbrace,
\end{equation}
where $C_{3,T}=2Ce^{\|b\|_{\infty}T}$ and $C_{0,T}=2e^{\|b\|_{\infty}T}$.
\end{prop}

\begin{proof}
By using the semigroup theory and the previous results, we obtain a unique regular solution $u \in C([0,T];H^3(\R))$. Therefore, we will prove that $u \in L^2(0,T,H^4(\R))$. Indeed, first note that $u_0 \in H^3(\R) \hookrightarrow L^2(\R)$, hence applying the Proposition \ref{prop5} it follows that $u \in B_{0,T}$ and
\begin{align}\label{ecc3}
\|u\|_{0,T}&\leq \leq C_T \left\lbrace \|u_0\|_{H^3(/R)} +\|f\|_{W^{1,1}(0,T; L^2(\R))}+\|f_x\|_{L^2(0,T; L^2(\R))} \right\rbrace
\end{align}
where $C_T= 2e^{\|b\|_{\infty}T}$. On the other hand, note that $u_t$ solves the problem
\begin{equation*}
\left\lbrace \begin{tabular}{l l}
$v_t -v_{xx}+v_{xxx}+bv=f_t$ & in $\R \times (0, \infty)$ \\
$v(0)=v_0$ & in $\R \times (0, \infty)$,
\end{tabular}\right.
\end{equation*}
where $v_0=\partial_x^2u_0-\partial_x^3u_0-bu_0+f(\cdot,0) \in L^2(\R)$. Then, by applying the Proposition \ref{prop5}, we have $u_t \in B_{0,T}$ and
\begin{align}\label{ecc4}
\|u_t\|_{0,T}&\leq _T \left\lbrace \|u_0\|_{H^3(\R)}+\|f(0)\|_{2} +\|f_t\|_{L^1(0,T; L^2(\R))}\right\rbrace,
\end{align}
obtaining (\ref{ecc2}). Moreover,
\begin{align}\label{ecc6}
\|(bu)_x\|_{L^2(0,T; L^2(\R))} &\leq \|b_x\|_2\|u\|_{L^2(0,T; L^{\infty}(\R))} +\|b\|_{\infty}\|u_x\|_{L^2(0,T; L^2(\R))} \notag \\
&\leq C \|b\|_{H^1(\R)}\|u\|_{0,T},
\end{align}
where $C$ is the embedding constant of $H^1(\R) \hookrightarrow L^{\infty}(\R)$. Since
\begin{equation*}\label{ecc5}
\partial_x^4u = \partial_x^3u -\partial_xu_t -\partial_x(bu) +\partial_xf \quad \text{in $D'(\R)$}, \quad \text{for all $t>0$},
\end{equation*}
we have that $\partial_x^4 u \in L^2(0,T;L^2(\R))$, i. e., $u\in L^2(0,T;H^4(\R))$ and $u \in B_{3,T}$. In order to prove (\ref{ecc1}), we need some estimates. Note that, from (\ref{ecc3}), we get
\begin{equation}\label{ecc7}
\sup_{t\in[0,T]}\|u(t)\|_2 \leq C_T \left\lbrace \|u_0\|_{H^3(\R)} +\|f\|_{W^{1,1}(0,T; L^2(\R))}+\|f_x\|_{L^2(0,T; L^2(\R))} \right\rbrace.
\end{equation}
Multiplying the equation in (\ref{ee1}) by $u_{xx}$ and integrating in $\R$ one obtains the inequality
\begin{align*}
\frac{1}{2}\frac{d}{dt}\|u_{x}(t)\|_2 +\|u_{xx}(t)\|_2\leq \left\lbrace \|f(t)\|_2 +\|bu(t)\|_2\right\rbrace \|u_{xx}(t)\|_2.
\end{align*}
Then, Young inequality leads to
\begin{align*}
\frac{1}{2}\frac{d}{dt}\|u_{x}(t)\|_2^2 +\frac{1}{2}\|u_{xx}(t)\|_2^2\leq C \left\lbrace \|f(t)\|_2^2 +\|b\|_{\infty}\|u(t)\|_2 ^2\right\rbrace
\end{align*}
Integrating on $[0,T]$, using (\ref{ecc3}) and the embedding $W^{1,1}(0,T)(0,T;L^2(\R))\hookrightarrow L^{\infty}(0,T;L^2(\R))$, the solution can be estimated as follows
\begin{equation}\label{ecc8}
\sup_{t\in[0,T]}\|u_{x}(t)\|_2 \leq C C_T \left\lbrace \|u_0\|_{H^3(\R)} +\|f\|_{W^{1,1}(0,T; L^2(\R))}+\|f_x\|_{L^2(0,T; L^2(\R))} \right\rbrace.
\end{equation}
A similar estimates is obtained by multiplying the equation  by $\partial_x ^4u$, integrating in $\R$ and using Young inequality:
\begin{align*}
\frac{1}{2}\frac{d}{dt}\|u_{xx}(t)\|_2^2 +\frac{1}{2}\|u_{xxx}(t)\|_2^2\leq C \left\lbrace \|f_x(t)\|_2^2 +\|(bu)_x(t)\|_2 ^2\right\rbrace.
\end{align*}
Integrating on $[0,T]$ and using (\ref{ecc6}) and (\ref{ecc3}), we have
\begin{equation}\label{ecc9}
\sup_{t\in[0,T]}\|u_{xx}(t)\|_2 \leq C C_T \left\lbrace \|u_0\|_{H^3(\R)} +\|f\|_{W^{1,1}(0,T; L^2(\R))}+\|f_x\|_{L^2(0,T; L^2(\R))} \right\rbrace.
\end{equation}
Since
\begin{align*}
\|u_{xxx}(t)\|_2 \leq \|u_t(t)\|_2+\|u_{xx}(t)\|_2 + \|bu(t)\|_2 +\|f(t)\|_2,
\end{align*}
using (\ref{ecc3}), (\ref{ecc4}), (\ref{ecc9}) and the embedding above, we conclude that
\begin{equation}\label{ecc10}
\sup_{t\in[0,T]}\|u_{xxx}(t)\|_2 \leq C C_T \left\lbrace \|u_0\|_{H^3(\R)} +\|f\|_{W^{1,1}(0,T; L^2(\R))}+\|f_x\|_{L^2(0,T; L^2(\R))} \right\rbrace.
\end{equation}
Putting together (\ref{ecc7}), (\ref{ecc8}), (\ref{ecc9}) and (\ref{ecc10}), we have
\begin{equation}\label{ecc11}
\|u\|_{C([0,T];H^3(\R))} \leq C C_T \left\lbrace \|u_0\|_{H^3(\R)} +\|f\|_{W^{1,1}(0,T; L^2(\R))}+\|f_x\|_{L^2(0,T; L^2(\R))} \right\rbrace.
\end{equation}
On the other hand,
\begin{align*}
\|\partial_x^4 u\|_{L^2(0,T;L^2(\R))} &\leq \|u_{xxx}\|_{L^2(0,T;L^2(\R))}+ \|\partial_xu_t\|_{L^2(0,T;L^2(\R))}+\|(bu)_x\|_{L^2(0,T;L^2(\R))}+\| f_x\|_{L^2(0,T;L^2(\R))} \\
&\leq T^{\frac{1}{2}}\|u\|_{C([0,T];H^3(\R))}+ \|u_t\|_{0,T}+\|(bu)_x\|_{L^2(0,T;L^2(\R))}+\| f_x\|_{L^2(0,T;L^2(\R))}.
\end{align*}
The above inequality, (\ref{ecc3}) - (\ref{ecc6}) and (\ref{ecc10}) allow us to conclude that
\begin{equation}\label{ecc12}
\|\partial_x^4 u\|_{L^2(0,T;L^2(\R))}  \leq C C_T \left\lbrace \|u_0\|_{H^3(\R)} +\|f\|_{W^{1,1}(0,T; L^2(\R))}+\|f_x\|_{L^2(0,T; L^2(\R))} \right\rbrace.
\end{equation}
(\ref{ecc11}) and (\ref{ecc12}) implies (\ref{ecc1}).
\end{proof}

\begin{thm}\label{teo5}
Let $b\in H^1(\R)$ and $a \in  C^2(\R)$ satisfying
\begin{gather}\label{e14}
|a(\mu)| \leq C(1+|\mu|^{p}), \quad  |a'(\mu)| \leq C(1+|\mu|^{p-1}) \quad \text{and} \quad  |a''(\mu)| \leq C(1+|\mu|^{p-2}), \qquad \forall \mu \in \R,
\end{gather}
with $p\geq 2$. Let $T>0$ and $u_0 \in H^3(\R)$. Then, there exist a unique solution $u \in B_{3,T}$ of $(\ref{e1})$, such that
\begin{equation*}\label{e15}
\|u\|_{3,T}\leq \eta_3(\|u_0\|_2)\|u_0\|_{H^3(\R)},
\end{equation*}
where $\eta_3: \R_+\rightarrow \R_+$ is a nondecreasing continuous function.
\end{thm}
-
\begin{proof}
Let $0<\theta \leq T$ and $R>0$ to be a constant to be determined later. Consider
\begin{equation*}\label{e16}
S_{\theta, R}:=\left\lbrace (u,u_t) \in B_{3,\theta}\times B_{0,\theta}:\|(u,u_t)\|_{B_{3,\theta}\times B_{0,\theta}}:= \|u\|_{{3,\theta}}+\|u_t\|_{ {0,\theta}} \leq R\right\rbrace.
\end{equation*}
Then, for each $(u,u_t) \in S_{\theta,R}\subset B_{3,\theta}\times B_{0,\theta}$, consider the problems
\begin{equation}\label{e17}
\left\lbrace \begin{tabular}{l}
$v_t = A_bv - a(u)u_x$ \\
$v(0)=u_0$
\end{tabular} \right.
\end{equation}
and
\begin{equation}\label{e18}
\left\lbrace \begin{tabular}{l}
$z_t = A_bz - [a(u)u_x]_x$ \\
$z(0)=z_0,$
\end{tabular} \right.
\end{equation}
with $z_0=-u_{0xxx}+u_{0xx}-bu_0-a(u_0)u_{0x} \in L^2(\R)$ and $A_bv=\partial_x^2v-\partial_x^3v-bv$. Recall that $A_b$ generates a strongly continuous semigroup $\{S(t)\}_{t\geq 0}$ of contractions in $L^2(\R)$. Moreover, by Lemma \ref{lm3} (i) and (ii), $a(u)u_x \in W^{1,1}(0,\theta;L^2(\R))$ and $[a(u)u_x]_x \in L^2(0,\theta;L^2(\R))$. Then, by Proposition \ref{prop8} the problems (\ref{e17}) and (\ref{e18}) have a unique mild solution $v$, such that $(v,v_t) \in B_{3,\theta}\times B_{0,\theta}$ and
\begin{equation}\label{e19}
\|(v,v_t)\|_{B_{3,\theta}\times B_{0,\theta}}\leq C_{\theta} \{ \|u_0\|_{H^3(\R)}+\|a(u)u_x\|_{W^{1,1}(0,\theta;L^2(\R))}+\|[a(u)u_x]_x\|_{L^2(0,\theta;L^2(\R))} \},
\end{equation}
where $C_{\theta}=2e^{\theta\|b\|_{\infty}}$. Thus, we can define the operator
\[
\Gamma:S_{\theta,R} \subset B_{3,\theta}\times B_{0,\theta}\longrightarrow B_{3,\theta}\times B_{0,\theta} \quad \text{by} \quad \Gamma (u,u_t) = (v,v_t).
\]
Since $C_{\theta}\leq C_T$, from (\ref{e19}) and Lemma $\ref{lm3}$, we have
\begin{align*}
\|\Gamma(u,u_t)\|_{B_{3,\theta}\times B_{0,\theta}}&\leq C_{T} \|u_0\|_{H^3(\R)}+C_TC\theta^{\frac{1}{2}}\left\lbrace (\|u\|_{3,\theta}+\|u_t\|_{0,\theta})+\|u\|_{3,\theta}^{p}(\|u\|_{3,\theta}+\|u_t\|_{0,\theta})+ \|u_t\|_{0,\theta}\|u\|_{3,\theta}\right. \\
&\left.+\|u_t\|_{0,\theta}\|u\|_{3,\theta}^{p}\right\rbrace + C_TC\theta^{\frac{1}{2}}\left\lbrace \|u\|_{3,\theta}^{2}+ 2\|u\|_{3,\theta}^{p+1}+\|u\|_{3,T}\right\rbrace \\
&\leq C_{T} \|u_0\|_{H^3(\R)} + C_TC\theta^{\frac{1}{2}}\left\lbrace 4R^{p+1}+2R^2+2R\right\rbrace.
\end{align*}
Choosing $R=2C_T\|u_0\|_{H^3(\R)}$, it follows that
\begin{equation*}\label{e21}
\|\Gamma(u,u_t)\|_{B_{3,\theta}\times B_{0,\theta}}\leq \left(K_1 +\frac{1}{2}\right)R,
\end{equation*}
where $K_1(\theta)= C_TC\theta^{\frac{1}{2}}\{4R^p+2R+2\}$. On the other hand, let $(u,u_t), (w,w_t) \in S_{\theta, R}$ and note that $\Gamma (u,u_t) - \Gamma (w,w_t)$ is solutions of
\begin{equation*}\label{e22}
\left\lbrace \begin{tabular}{l}
$v_t = A_bv +[a(w)w_x- a(u)u_x]$ \\
$v(0)=0$
\end{tabular} \right.
\end{equation*}
and
\begin{equation*}\label{e23}
\left\lbrace \begin{tabular}{l}
$z_t = A_bz + [a(w)w_x-a(u)u_x]_x$ \\
$z(0)=0.$
\end{tabular} \right.
\end{equation*}
Hence, from Lemma $\ref{lm3}$, the following estimate holds
\begin{align}\label{e24}
\|\Gamma(u,u_t)-\Gamma(w,w_t)\|_{B_{3,\theta}\times B_{0,\theta}} &\leq C_{T} \{\|a(w)w_x-a(u)u_x\|_{W^{1,1}(0,\theta;L^2(\R))}+\|[a(w)w_x-a(u)u_x]_x\|_{L^2(0,\theta;L^2(\R))} \}.
\end{align}
The next steps are devoted to estimate the terms on the right hand side of (\ref{e24}):
\begin{align*}
\|a(w)w_x-a(u)u_x\|_{W^{1,1}(0,\theta;L^2(\R))} &\leq \|(a(w)-a(u))w_x\|_{W^{1,1}(0,\theta;L^2(\R))}+\|a(u)(w-u)_x\|_{W^{1,1}(0,\theta;L^2(\R))}.
\end{align*}
By using the Mean Valued Theorem and Lemma \ref{lm3}, we have
\begin{align}\label{e26}
\|a(w)w_x-a(u)u_x\|_{W^{1,1}(0,\theta;L^2)} &\leq C\|(1+|u|^{p-1}+|w|^{p-1})|w-u|w_x\|_{W^{1,1}(0,\theta;L^2(\R))} \notag\\
&+\|a(u)(w-u)_x\|_{W^{1,1}(0,\theta;L^2(\R))} \notag\\
&\left.+\||w|^{p-1}|w-u|w_x\|_{W^{1,1}(0,\theta;L^2(\R))}\right\rbrace+\|a(u)(w-u)_x\|_{W^{1,1}(0,\theta;L^2(\R))} \notag\\
&\leq C\left\lbrace  2^{\frac{1}{2}}\theta^{\frac{1}{4}}\left\lbrace \|w-u\|_{3,\theta}\|w\|_{3,\theta}+ \|(w-u)_t\|_{0,\theta}\|w\|_{3,\theta}+\|w-u\|_{3,\theta}\|w_t\|_{0,\theta}\right\rbrace \right. \notag\\
&+ \theta^{\frac{1}{2}}\{\|w-u\|_{3,\theta}\|u\|_{3,\theta}^{p-1}(\|w\|_{3,\theta}+\|w_t\|_{0,\theta})+\|(w-u)_t\|_{0,\theta}\|w\|_{3,\theta}\|u\|_{3,\theta}^{p-1} \notag\\
&+ (p-1)\|u_t\|_{0,\theta}\|w-u\|_{3,\theta}\|u\|_{3,\theta}^{p-2}\|w\|_{3,\theta}\} \notag \\
&+ \theta^{\frac{1}{2}}\{ \|w-u\|_{3,\theta}\|w\|_{3,\theta}^{p-1}(\|w\|_{3,\theta}+\|w_t\|_{0,\theta}) \notag \\
&\left.+\|w\|_{3,\theta}^p\|[w-u]_t\|_{0,\theta}+(p-1)\|w-u\|_{3,\theta}\|w\|_{3,\theta}^{p-1}\|w_t\|_{0,\theta}\}\right\rbrace \notag\\
&+C\theta^{\frac{1}{2}} \left\lbrace (\|w-u\|_{3,\theta}+\|(w-u)_t\|_{0,\theta}) +\|u\|_{3,\theta}^{p}(\|w-u\|_{3,\theta}+\|(w-u)_t\|_{0,\theta}) \right. \notag \\
&\left.+\|u_t\|_{0,\theta}\|w-u\|_{3,\theta}+\|u_t\|_{0,\theta}\|w-u\|_{3,\theta}\|u\|_{3,\theta}^{p-1}\right\rbrace \notag\\
&\leq K_2 \|(w-u,(w-u)_t)\|_{B_{3,\theta}\times B_{0,\theta}}
\end{align}
where $K_2(\theta)=C\left\lbrace (2^{\frac{3}{2}}\theta^{\frac{1}{4}}+\theta^{\frac{1}{2}})R+2(p+1)\theta^{\frac{1}{2}}R^p+\theta^{\frac{1	 }{2}}\right\rbrace$. To estimate the second term, note that
\begin{align}\label{e31}
[a(w)w_x-a(u)u_x]_x =[a'(w)-a'(u)]w_x^2+a'(u)[w-u]_x[w+u]_x+[a(w)-a(u)]w_{xx}+a(u)[w_{xx}-u_{xx}]
\end{align}
Then, from the Mean Valued Theorem, (\ref{e25}) and (\ref{e14}), we have the following estimates:
\begin{align*}
 \|[a'(w)-a'(u)]w_x^2\|_{L^2(0,\theta;L^2(\R))} &\leq C  \|[1+|w|^{p-2}+|u|^{p-2}]|w-u|w_x^2\|_{L^2(0,\theta;L^2(\R))} \\
&\leq C\{ \||w-u|w_x^2\|_{L^2(0,\theta;L^2(\R))} + \||w|^{p-2}|w-u|w_x^2\|_{L^2(0,\theta;L^2(\R))} \\
&+\||u|^{p-2}|w-u|w_x^2\|_{L^2(0,\theta;L^2(\R))} \}\\
&\leq C\theta^{\frac{1}{2}}\left\lbrace \|w\|_{3,\theta}^2\|w-u\|_{3,\theta}+\|w\|_{3,\theta}^{p}\|w-u\|_{3,\theta}+\|u\|_{3,\theta}^{p-2}\|w\|_{3,\theta}^2\|w-u\|_{3,\theta} \right\rbrace,
\end{align*}
\begin{align*}
\|a'(u)[w-u]_x[w+u]_x\|_{L^2(0,\theta;L^2(\R))} &\leq \|a'(u)[w-u]_xw_x\|_{L^2(0,\theta;L^2(\R))}+\|a'(u)[w-u]_xu_x\|_{L^2(0,\theta;L^2(\R))} \\
&\leq C\{ \|[w-u]_xw_x\|_{L^2(0,\theta;L^2(\R))}+\|[w-u]_xu_x\|_{L^2(0,\theta;L^2(\R))}  \\
&+\|u|^{p-1}|[w-u]_xw_x\|_{L^2(0,\theta;L^2(\R))}+\||u|^{p-1}[w-u]_xu_x\|_{L^2(0,\theta;L^2(\R))} \}  \\
&\leq C\theta^{\frac{1}{2}}\left\lbrace  \|w\|_{3,\theta}\|w-u\|_{3,\theta}+  \|u\|_{3,\theta}\|w-u\|_{3,\theta}  \right. \\
&\left. + \|u\|_{3,\theta}^{p-1}\|w\|_{3,\theta}\|w-u\|_{3,\theta}+\|u\|_{3,\theta}^{p}\|w-u\|_{3,\theta} \right\rbrace
\end{align*}
and
\begin{align*}
\|[a(w)-a(u)]w_{xx}\|_{L^2(0,\theta;L^2(\R))} \leq C\theta^{\frac{1}{2}}\left\lbrace  \|w\|_{3,\theta}\|w-u\|_{3,\theta}+ \|u\|_{3,\theta}^{p-1}\|w\|_{3,\theta}\|w-u\|_{3,\theta}+\|w\|_{3,\theta}^{p}\|w-u\|_{3,\theta} \right\rbrace
\end{align*}
\begin{align*}
\|a(u)[w_{xx}-u_{xx}]\|_{L^2(0,\theta;L^2(\R))} \leq C\theta^{\frac{1}{2}}\left\lbrace \|w-u\|_{3,\theta} + \|u\|_{3,\theta}^{p}\|w-u\|_{3,\theta} \right\rbrace.
\end{align*}
The above estimates and (\ref{e31}), show that
\begin{align}\label{e32}
\|[a(w)w_x-a(u)u_x]_x\|_{L^2(0,\theta;L^2(\R))} &\leq C\theta^{\frac{1}{2}}\left\lbrace 2\|w\|_{3,\theta}^{p}+ \|w\|_{3,\theta}^{2}+2 \|w\|_{3,\theta}+2 \|u\|_{3,\theta}^{p}+\|u\|_{3,\theta}\right. \notag \\
&\left.+2\|u\|_{3,\theta}^{p-1}\|w\|_{3,\theta}+\|u\|_{3,\theta}^{p-2}\|w\|_{3,\theta}^2+1\right\rbrace \|w-u\|_{3,\theta} \notag \\
&\leq K_3 \|(w-u,[w-u]_t)\|_{B_{3,\theta}\times B_{0,\theta}},
\end{align}
where $K_3(\theta) = C\theta ^{\frac{1}{2}}\left\lbrace 7R^p+R^2+3R +1 \right\rbrace$. From (\ref{e24}), (\ref{e26}) and (\ref{e32}), we get
\begin{equation*}
\|\Gamma(u,u_t)-\Gamma(w,w_t)\|_{B_{3,\theta}\times B_{0,\theta}} \leq K_4  \|(w-u,[w-u]_t)\|_{B_{3,\theta}\times B_{0,\theta}},
\end{equation*}
where $K_4 = C_T(K_2+K_3)=C_TC\theta^{\frac{1}{2}}\left\lbrace 2(p+9)R^p+R^{p-1}+R^2+4R+2\right\rbrace+2^{\frac{3}{2}}C_TC\theta^{\frac{1}{4}}R$. Note that $K_1 \leq K_4$, therefore choosing $\theta >0$, such that $K_4<\frac{1}{2}$, it follows that
\begin{equation*}\label{e33}
\left\lbrace \begin{tabular}{l}
$\|\Gamma(u,u_t)\|_{B_{3,\theta}\times B_{0,\theta}} \leq R$ \\
$\|\Gamma(u,u_t)-\Gamma(w,w_t)\|_{B_{3,\theta}\times B_{0,\theta}} \leq \frac{1}{2}\|(w-u,[w-u]_t)\|_{B_{3,\theta}\times B_{0,\theta}}$
\end{tabular}\right., \quad \forall (u,u_t),(w,w_t) \in S_{\theta,R} \subset B_{3,\theta}\times B_{0,\theta}.
\end{equation*}
Hence $\Gamma : S_{\theta,R}\longrightarrow S_{\theta,R}$ is a contraction and, by Banach fixed point theorem, we obtain a unique $(u,u_t) \in S_{\theta,R}$, such that $\Gamma (u,u_t)=(u,u_t)$. Thus, $u$ is a unique local mild solution to problem (\ref{e1}) and satisfies
\begin{equation}\label{e34}
\|u\|_{3,\theta} \leq 2C_T\|u_0\|_{H^3(\R)}.
\end{equation}
Moreover, (\ref{e34}) implies the solution does not blow-up in finite time and, by using standards arguments we can extent $\theta$ to $T$. Finall, the proof is complete defining $\eta_3(s)=2C_T$.
\end{proof}

\section{Exponencial stability}

This section is devoted to prove the exponential decay of the solutions under the assumptions (\ref{hyp b}) and (\ref{hyp a}). We consider two cases: $1 \leq p < 2$ and $2\leq p < 5$.
\subsection{\texorpdfstring{Case $1 \leq p <2$.}{Case 1}}
In order to make our work self-contained, we prove the following proposition which is simliar to Theorem 5.1 in \cite{cavalcanti2014global}.

\begin{prop}\label{prop6}
Let $b$ satisfying (\ref{hyp b}). Then, for any $u_0 \in L^2(\R)$ and $1\leq p < 2$, the corresponding solution $u$ of (\ref{e1}) is exponential stable and satisfies the decay estimate
\begin{equation}\label{e119}
\|u(t)\|_{2}\leq e^{-2\lambda_0 t}\|u_0\|_2,\quad \forall t\geq 0.
\end{equation}
\end{prop}

\begin{proof}
We first consider $u_0  \in H^3(\R)$ and $u$ corresponding smooth solution. Multiplying the equation in (\ref{e1}) by $u$ and integrating in $\R$, we have
\begin{equation*}
\frac{d}{dt}\|u(t)\|_2^2+2\|u_x(t)\|^2_2=-2\int_{\R}b(x)|u(x,t)|^2dx.
\end{equation*}
Hence, Proceeding as in \cite[Theorem 5.1]{cavalcanti2014global}, we obtain
\begin{equation*}
\|u(t)\|_{2}\leq e^{-2\lambda_0 t}\|u_0\|_2.
\end{equation*}
Now, let $u_0 \in L^2(\R)$ and $u$ the corresponding mild solution given by Theorem \ref{teo1}. Consider $\{u_{n,0}\} \in H^3(\R)$, such that
\begin{equation*}
u_{n,0} \rightarrow u_0 \quad \text{in $L^2(\R)$}.
\end{equation*}
Then, the corresponding strong solutions $u_n$ satisfy the estimate
\begin{equation}\label{e120}
\|u_n(t)\|_{2}\leq e^{-2\lambda_0 t}\|u_{n,0}\|_2.
\end{equation}
On the other hand, note that the identity (\ref{e118}) in Theorem \ref{teo1} implies that
\begin{equation*}
u_n \rightarrow u \quad \text{in $L^2(\R)$}.
\end{equation*}
Taking the limit in (\ref{e120}), we obtain (\ref{e119}).
\end{proof}

\begin{cor}\label{cor2}
Let $T >0$, $u_0 \in L^2(\R)$ and $b$  satisfying (\ref{hyp b}). Then, there exist a nondecreasing continuous function $\alpha_0:\R^+ \rightarrow \R^+$, such that the corresponding solution $u$ of problem (\ref{e1}) with $1\leq p< 2$ satisfies
\begin{equation*}\label{e121}
\|u\|_{0,[t,t+T]}\leq \alpha_0(\|u_0\|_2)e^{-2\lambda_0t}, \quad \forall t\geq 0.
\end{equation*}
\end{cor}

\begin{proof}
Note that, after a change of variable, the restriction of $u$ to $[t,t+T]$ is a solution of problem (\ref{e1}) with respect to the initial data $u(t)$. Then, by Theorem \ref{teo2} and Proposition \ref{prop6} we have
\begin{equation*}
\|u\|_{0,[t,t+T]}\leq \beta_0(\|u(t)\|_2)\|u(t)\|_2 \leq \beta_0(e^{-2\lambda_0 t}\|u_0\|_2)\|u_0\|_2 e^{-2\lambda_0 t} \leq \alpha_0(\|u_0\|_2)e^{-2\lambda_0 t},
\end{equation*}
where $\alpha_0(s)=\beta_0(s)s$.
\end{proof}

The next result was inspired by the ideas introduced in the proof of Theorem 6.1 in \cite{cavalcanti2014global} and Proposition 3.9 in \cite{rosier2006global}.

\begin{prop}\label{prop7}
Let $T>0$,  $1\leq p < 2$, $a(0)=0$ and $b$ satisfying (\ref{hyp b}).  Then, there exist $\gamma>0$, $T_0>0$ and a nonnegative continuous  function $\alpha_3: \R^+ \rightarrow \R^+$, such that, for every $u_0 \in H^3(\R)$, the corresponding solution $u$ satisfies
\begin{equation}\label{e129}
\|u(t)\|_{H^3(\R)} \leq \alpha_3(\|u_0\|_2,T_0)\|u_0\|_{H^3(\R)}e^{-\gamma t}, \qquad \forall t\geq T_0.
\end{equation}
\end{prop}

\begin{proof}
Let $v=u_t$. Then, by Proposition \ref{prop4} $v$ solves linearized equation (\ref{ee105}) with $v_0= -\partial_x^3u_0+\partial_x^2u_0-a(u_0)\partial_xu_0-bu_0$ and satisfies
\begin{equation}\label{e122}
\|v\|_{0,T}\leq \sigma(\|u\|_{0,T})\|v_0\|_2.
\end{equation}
After a change of variable, the restriction of $v$ to [t, t + T] is a solution of problem (\ref{ee105}) with respect to the initial data $v(t)$ and
\begin{equation*}
\|v\|_{0,[t,t+T]}\leq \sigma(\|u\|_{0,[t,t+T]})\|v(t)\|_2.
\end{equation*}
Applying Corollary \ref{cor2}, it follows that
\begin{equation}\label{e125}
\|v\|_{0,[t,t+T]}\leq \sigma(\alpha_0(\|u_0\|_2)e^{-2\lambda_0t})\|v(t)\|_2 \leq \sigma(\alpha_0(\|u_0\|_2))\|v(t)\|_2.
\end{equation}
On the other hand, the solution $v$ may be written as
\begin{equation*}
v(t)=S(t)v_0-\int_0^t S(t-s)[a(u(s))v(s)]_xds
\end{equation*}
where $S(t)$ is a $C_0$-semigroup of contraction in $L^2(\R)$ generated by the operator $A_b$. Note that $v_1(t)=S(t)v_0$ is solution of the problem (\ref{ee105}) with $a(u)=0$. Then, proceeding as in the proof of Proposition \ref{prop6}, we have
\begin{equation}\label{e123}
\|v_1(t)\|_2 \leq \|v_0\|e^{-2\lambda_0 t}, \qquad \forall t \geq 0.
\end{equation}
Let us now denote $v_2(t)=\int_0^t S(t-s)[a(u(s))v(s)]_xds$. Note that
\begin{align*}
 \|v_2(T)\|_2 &\leq \|a'(u)u_xv\|_{L^1(0,T;L^2(\R))}+ \|a(u)v_x\|_{L^1(0,T;L^2(\R))}.
\end{align*}
Moreover, $a(0)=0$ implies that $|a(u)| \leq C (1+|u|^{p-1})|u|$, for some $C >0$. Thus, by using  Lemma \ref{lm2},  the following holds
\begin{align}\label{e124}
\|v_2(T)\|_2 &\leq  C\left\lbrace \|(1+|u|^{p-1})u_xv\|_{L^1(0,T;L^2(\R))}+ \|(1+|u|^{p-1})|u|v_x\|_{L^1(0,T;L^2(\R))} \right\rbrace \notag \\
&\leq  2C\left\lbrace 2^{\frac{1}{2}}T^{\frac{1}{4}}\|u\|_{0,T}\|v\|_{0,T}+ 2^{\frac{p}{2}}T^{\frac{2-p}{4}}\|u\|_{0,T}^p\|v\|_{0,T}\right\rbrace.
\end{align}
Using (\ref{e122}), (\ref{e123}) and (\ref{e124}), we obtain a positive constant $K_T$, such that
\begin{align*}
\|v(T)\|_2 \leq \left(e^{-2\lambda_0T}+K_T(1+\|u\|_{0,T}^{p-1})\|u\|_{0,T}\sigma(\|u\|_{0,T})\right)\|v_0\|_2.
\end{align*}
With the notation introduced above, we consider the sequence $y_n(\cdot)=v(\cdot, nT)$ and introduce $w_n(\cdot,t)=v(\cdot,t+nT)$. For $t\in [0,T]$, $w_n$ solves the problem
\begin{equation*}
\left\lbrace\begin{tabular}{l l}
$\partial_t w_n+\partial_x^3 w_n-\partial_x^2 w_n+[a(u(\cdot+nT))w_n]_x+bw_n= 0$ & in $\R\times \R^+$ \\
$w_n(0)=y_n$ & in $\R$.
\end{tabular}\right.
\end{equation*}
First, observe that we can obtain for $y_n$ a estimate similar to the one obtained for $v(T)$:
\begin{align}\label{yn}
\|y_{n+1}\|_2&=\|w_{n}(T)\|_2 \leq e^{-2\lambda_0T}\|w_0\|_2+K_T(1+\|u(\cdot+nT)\|_{0,T}^{p-1})\|u(\cdot+nT)\|_{0,T} \|w_n\|_{0,T} \notag \\
&\leq \left\lbrace e^{-2\lambda_0T}+K_T(1+\|u\|_{0,[nT,(n+1)T]}^{p-1})\|u\|_{0,[nT,(n+1)T]}\sigma(\|u\|_{0,[nT,(n+1)T]})\right\rbrace\|y_n\|_2.
\end{align}
On the other hand, we can take $\beta >0$, small enough, such that
\begin{equation*}
e^{-2\lambda_0T}+K_T(1+\beta^{p-1})\beta\sigma(\beta) < 1.
\end{equation*}
Whit this choice of $\beta$, Corollary \ref{cor2} allows us to choose $N>0$, large enough, satisfying
\begin{equation*}
\|u\|_{0,[nT,(n+1)T]}\leq \alpha_0(\|u_0\|_2)e^{-2\lambda_0 nT} \leq \alpha_0(\|u_0\|_2)e^{-2\lambda_0 NT} \leq \beta, \qquad \forall n> N.
\end{equation*}
Thus, from (\ref{yn}) we obtain the following estimate
\begin{equation*}
\|y_{n+1}\|_2\leq r \|y_n\|_2, \qquad \forall n\geq N, \quad \text{where $0<r<1$},
\end{equation*}
which implies
\begin{equation}\label{e126}
\|v((n+k)T)\|_2\leq r^k \|v(nT)\|_2, \qquad \forall n\geq N.
\end{equation}
Let $T_0=NT$ and $t \geq T_0$. Then, there exists $k \in \N$ and $\theta \in [0,T]$, satisfying
\begin{equation*}
t=(N+k)T+\theta.
\end{equation*}
Then, from (\ref{e125}) and (\ref{e126}), it is found that
\begin{align*}
\|v(t)\|_2& \leq \|v\|_{0,[(N+k)T, (N+k+1)T]} \leq \sigma (\alpha_0(\|u_0\|_2))\|v((N+k)T)\|_2 \\
&\leq \sigma (\alpha_0(\|u_0\|_2))r^{\frac{t-NT-\theta}{T}}\|v(T_0)\|_2 \\
&\leq \sigma (\alpha_0(\|u_0\|_2))r^{\frac{t-NT-\theta}{T}}\sigma (\alpha_0(T_0,\|u_0\|_2))\|v(0)\|_2 \\
&\leq \eta_1(\|u_0\|)e^{-\delta_1t}\|v_0\|_2,
\end{align*}
where $\delta_1=\frac{1}{T}\ln\left(\frac{1}{r}\right)$ and $\eta_1(s)=\sigma(\alpha_0(s))\sigma(\alpha_0(T_0,s))r^{-(N+1)}$. Invoking the estimate (\ref{e127}) in Theorem \ref{prop3}, and having in mind that $v=u_t$, we get
\begin{equation}\label{e128}
\|u_t(t)\|_2\leq \eta_2(\|u_0\|_2)\|u_0\|_{H^3(\R)}e^{-\delta_1t}, \qquad \forall t \geq T_0,
\end{equation}
where $\eta_2(s)=\eta_1(s)C(s)$. On the other hand, note that
\begin{equation}\label{exxx}
\|u_{xxx}(t)\|_2\leq \|u_t(t)\|_2+\|u_{xx}(t)\|_2+\|a(u(t))u(t)\|_2+\|b\|_{\infty}\|u(t)\|_2.
\end{equation}
Estimating the nonlinear term as in the proof of Lemma \ref{lm1},
\begin{align*}
\|a(u(t))u(t)\|_2= \|u(t)^{p+1}u_x(t)\|_2 \leq \|u(t)^{p+1}\|_{\infty}\|u_x(t)\|_2 \leq 2^{\frac{p+1}{2}}\|u(t)\|_2^{\frac{p+1}{2}}\|u_x(t)\|_2^{\frac{p+3}{2}},
\end{align*}
from (\ref{exxx}) we obtain
\begin{align*}
\|u_{xxx}(t)\|_2\leq \|u_t(t)\|_2+\|u_{xx}(t)\|_2+2^{\frac{p+1}{2}}\|u(t)\|_2^{\frac{p+1}{2}}\|u_x(t)\|_2^{\frac{p+3}{2}}+\|b\|_{\infty}\|u(t)\|_2.
\end{align*}
Using Gagliardo-Nirenberg and Young inequalities, it follows that
\begin{align*}
\|u_{xxx}(t)\|_2 &\leq \|u_t(t)\|_2+C\|u_{xxx}(t)\|_2^{\frac{2}{3}}\|u(t)\|_2^{\frac{1}{3}}+2^{\frac{p+1}{2}} C\|u(t)\|_2^{\frac{5p+9}{6}}\|u_{xxx}(t)\|_2^{\frac{p+3}{6}}+\|b\|_{\infty}\|u(t)\|_2 \\
&\leq \|u_t(t)\|_2+ \left(\frac{C}{3\varepsilon}+\|b\|_{\infty}\right)\|u(t)\|_2+\frac{2^{\frac{p+1}{2}}(3-p)C}{6\varepsilon}\|u(t)\|_2^{\frac{5p+9}{3-p}} + \left(\frac{p+7}{6}\right)C\varepsilon\|u_{xxx}(t)\|_2.
\end{align*}
Choosing $\varepsilon=\frac{3}{C(p+7)}$, we have
\begin{align*}
\|u_{xxx}(t)\|_2 &\leq 2\|u_t(t)\|_2+2\left( \frac{C^2(p+7)}{9}+\|b\|_{\infty}\right)\|u(t)\|_2+\frac{2^{\frac{p+1}{2}}(3-p)(p+7)C^2}{18}\|u(t)\|_2^{\frac{5p+9}{3-p}}.
\end{align*}
Applying Proposition \ref{prop6} and estimate (\ref{e128}), the following decay estimate holds
\begin{equation}\label{ec1}
\|u_{xxx}(t)\|_2\leq \eta_3(\|u_0\|)\|u_0\|_{H^3(\R)}e^{-\gamma t}, \qquad \forall t \geq T_0,
\end{equation}
where $\eta_3(s)=2\eta_2(s)+\frac{2}{9}C^2(p+7)+2\|b\|_{\infty}+\frac{1}{3}2^{\frac{p+1}{2}}C^2(3-p)(p+7)s^{\frac{6p-1}{3-p}}$ and $\gamma=\min\left\lbrace\delta_1,2\lambda_0\right\rbrace$.
Now, using Gagliardo-Nirenberg and Young inequalities it is easy to obtain
\begin{equation*}
\|u(t)\|_{H^3(\R)}\leq C_1\left(\|u(t)\|_2+\|u_{xxx}(t)\|_2\right).
\end{equation*}
Finally, by Proposition \ref{prop6} and (\ref{ec1}) we obtain (\ref{e129}) with $\alpha_3(s)=C_1\left(1+\eta_3(s)\right)$.
\end{proof}

Proposition \ref{prop6} and \ref{prop7}, together with Corollary \ref{cor1} and interpolation arguments give the main result of this section.
%%%%%%%%%%%%%%%%%%%%%%%%%%%%%%%%%%%%%%%%%%%%%%%%%%%%%%%%%%%%%%
\begin{thm}\label{teo3}
Let $T>0$,  $1\leq p < 2$, $a(0)=0$ and $b$ satisfying (\ref{hyp b}).  Then, there exist positive constants $\gamma$, $\varepsilon_0$ and a continuous nonnegative function $\alpha: \R^+ \rightarrow \R^+$, such that, for every $u_0 \in H^s(\R)$, with $0\leq s\leq 3$,  the corresponding solution $u$ satisfies
\begin{equation}\label{e130}
\|u(t)\|_{H^s(\R)} \leq \alpha(T_0,\|u_0\|_2)\|u_0\|_{H^s(\R)}e^{-\lambda t}, \qquad \forall t\geq T_0.
\end{equation}
\end{thm}

\begin{proof}
By Corollary \ref{cor1} the corresponding solution $u$ belongs to $B_{0,[\varepsilon,T]}$, for all $\varepsilon \in (0,T]$. In particular, we choose $\varepsilon \leq T_0$, where $T_0$ is given by Proposition \ref{prop7}. Then, by using the interpolation inequality (2.43) in \cite[pag. 19]{lions1972non}, we have
\begin{equation*}
\|u(t)\|_{H^s(\R)}=\|u(t)\|_{[L^2(\R),H^3(\R)]_{2,\frac{s}{3}}} \leq C\|u(t)\|_2^{1-\frac{s}{3}}\|u(t)\|_{H^3(\R)}^{\frac{s}{3}}, \quad \forall t\geq \varepsilon.
\end{equation*}
Finally, Propositions \ref{prop6} and \ref{prop7} give us that
\begin{equation*}
\|u(t)\|_{H^s(\R)}\leq Ce^{-2\left(1-\frac{s}{3}\right)\lambda_0t}\|u_0\|_2^{\left(1-\frac{s}{3}\right)}\alpha_3^{\frac{s}{3}}(\|u_0\|_2,T_0)e^{-\frac{s}{3}\gamma t}, \qquad \forall t \geq T_0.
\end{equation*}
Observe that, by construction, $\gamma \leq 2\lambda_0$, therefore we obtain (\ref{e130}) with $\alpha(s,T_0)=C\alpha_3^{\frac{s}{3}}(s,T_0)$.
\end{proof}

\subsection{\texorpdfstring{Case $2 \leq p <5$.}{Case 2}}
Along this section we assume that the damping function $b = b(x)$ does not change sign and satisfies (\ref{hyp a}). Under this condition, we prove the exponential decay of the solutions in the $L^2$-norm by using the so-called compactness-uniqueness argument. The key is to establish the unique continuation property for the solution of the GKdV-B equation. The proof of this unique continuation property is mainly based on a Carleman estimate.

The next Carleman estimate is based on the global Carleman inequality obtained for the KdV equation in \cite{rosier2000exact}. The proof is given in the Appendix.

\begin{lem}[\textbf{Carleman's estimative}]\label{carl}
Let $T$ and $L$ be positive numbers. Then, there exist a smooth positive function $\psi$ on $[-L, L]$ (which depends on L) and positive constants $s_0 = s_0(L, T)$ and $C = C(L, T)$, such that, for all $s \geq s_0$ and any
\begin{equation}\label{e131}
q \in L^2(0,T;H^3(-L,L))\cap H^1(0,T;L^2(-L,L))
\end{equation}
satisfying
\begin{equation}\label{e132}
\text{$q(t, \pm L) = q_x(t, \pm L) = q_{xx}(t, \pm L) = 0$, for $0 \leq t \leq  T$},
\end{equation}
we have
\begin{gather*}
\int_{0}^{T}\int_{-L}^{L}\left\lbrace \frac{s^5}{t^5(T-t)^5}|q|^2+\frac{s^3}{t^3(T-t)^3}|q_x|^2+\frac{s}{t(T-t)}|q_{xx}|^2\right\rbrace e ^{-\frac{2s\psi(x)}{t(T-t)}}dxdt \notag \\
\leq C \int_{0}^{T}\int_{-L}^{L}|q_t-q_{xx}+q_{xxx}|^2e ^{-\frac{2s\psi(x)}{t(T-t)}}dxdt.
\end{gather*}
\end{lem}

\begin{lem}[\textbf{Unique continuation property}]\label{ucp}
Let $T$ be a positive number. If $u  \in L^{\infty}(0,T;H^1(\R))$ solves
\begin{equation*}\label{e142}
\left\lbrace
\begin{tabular}{l l}
$u_t -u_{xx}+u_{xxx}+a(u)u_x=0$ & in $\R \times (0,T)$ \\
$u \equiv 0$ & in $(-\infty,-L)  \cup (L,\infty) \times (0,T),$ \\
\end{tabular}\right.
\end{equation*}
for some $L>0$, with $a \in C(\R)$ satisfying (\ref{e2}), then $u \equiv 0$ in $\R \times (0,T)$.
\end{lem}

\begin{proof}
For $h>0$, consider
\begin{align*}
u^h(x,t)=\frac{1}{h}\int_t^{t+h}u(x,s)ds.
\end{align*}
Then, $u^h \in W^{1,\infty}(0,T',H^1(\R))$ and
\begin{equation}\label{e143}
u^h \rightarrow u \quad \text{in $L^{\infty}(0,T';H^1(\R))$},
\end{equation}
for any $T'<T$. Moreover, $u^h$ solves
\begin{equation}\label{e144}
\left\lbrace
\begin{tabular}{l l}
$u_t^h -u^h_{xx}+u^h_{xxx}+(a(u)u_x)^h=0$ & in $\R \times (0,T')$ \\
$u^h \equiv 0$ & in $(-\infty,-L)  \cup (L,\infty)\times (0,T)$. \\
\end{tabular}\right.
\end{equation}
On the other hand, note that $u \in L^{\infty}(0,T,H^1(\R))$ implies $a(u)u_x \in L^{\infty}(0,T,L^2(\R))$. Indeed, since
\begin{align*}
\|a(u)u_x\|_{L^{\infty}(0,T,L^2(\R))}\leq C\left\lbrace \|u\|_{L^{\infty}(0,T,H^1(\R))}+\|u\|^p_{L^{\infty}(0,T,L^{\infty}(\R))}\|u\|_{L^{\infty}(0,T,H^1(\R))}\right\rbrace,
\end{align*}
$(a(u)u_x)^h \in L^{\infty}(0,T,L^2(\R))$.  Then, proceeding as in the proof of Theorem \ref{prop3}, we have
\begin{equation*}
u^h \in L^{\infty}(0,T',H^3_0(-L,L)) \cap  H^1(0,T',L^2(-L,L)).
\end{equation*}
Invoking the Lemma \ref{carl} we obtain $C, s_0 > 0$ and a positive function $\psi$, such that
\begin{align*}
\int_Q\left\lbrace \frac{s^5|u^h|^2}{t^5(T-t)^5}+\frac{s^3|u^h_x|^2}{t^3(T-t)^3}+\frac{s|u^h_{xx}|^2}{t(T-t)}\right\rbrace e ^{-\frac{2s\psi(x)}{t(T-t)}}dxdt \leq C \int_Q|u^h_t-u^h_{xx}+u^h_{xxx}|^2e ^{-\frac{2s\psi(x)}{t(T-t)}}dxdt,
\end{align*}
for all $s >s_0$ and $Q=(0,T')\times (-L,L)$. By (\ref{e144}),
\begin{align*}
\int_Q|u^h_t-u^h_{xx}+u^h_{xxx}|^2e ^{-\frac{2s\psi(x)}{t(T-t)}}dxdt
&=  \int_Q|(a(u)u_x)^h|^2e ^{-\frac{2s\psi(x)}{t(T-t)}}dxdt\\
&\leq  \int_Q|a(u)u^h_x|^2e ^{-\frac{2s\psi(x)}{t(T-t)}}dxdt+\int_Q|(a(u)u_x)^h-a(u)u^h_x|^2dxdt \\
&\leq \|a(u)\|^2_{L^{\infty}(Q)}\int_Q|u^h_x|^2e ^{-\frac{2s\psi(x)}{t(T-t)}}dxdt+\|(a(u)u_x)^h-a(u)u^h_x\|^2_{L^2(Q)}.
\end{align*}
Hence,
\begin{gather*}\label{e145}
0<\int_Q\left\lbrace \frac{s^5}{t^5(T-t)^5}|u^h|^2+\left(\frac{s^3}{t^3(T-t)^3}-C\|a(u)\|^2_{L^{\infty}(Q)}\right)|u^h_x|^2+\frac{s}{t(T-t)}|u^h_{xx}|^2\right\rbrace e ^{-\frac{2s\psi(x)}{t(T-t)}}dxdt \notag\\
\leq C \|(a(u)u_x)^h-a(u)u^h_x\|^2_{L^2(Q)},
\end{gather*}
since, for $s$, large enough, we obtain $\frac{s^3}{t^3(T-t)^3}-C\|a(u)\|^2_{L^{\infty}(Q)} >0$.\\

Note that (\ref{e143}) guarantees that
%implies that $u^h \rightarrow u$ in $L^2(0,T;L^2(-L.L))$  and
$a(u)u^h_x \rightarrow a(u)u_x$ in $L^2(0,T;L^2(-L.L))$, since $a(u) \in L^{\infty}(0,T',L^{\infty}(-L,L))$. Moreover, as $a(u)u_x \in L^{2}(0,T',L^{2}(-L,L))$ we have that $(a(u)u_x)^h \in W^{1,\infty}(0,T',L^{2}(-L,L))$ and $(a(u)u_x)^h \rightarrow a(u)u_x \in L^{2}(0,T',L^{2}(-L,L))$. %Then $(a(u)u_x)^h \rightarrow a(u)u_x \in L^{2}(0,T',L^{2}(-L,L))$.
Thus, passing to the limit in (\ref{e145}), we obtain that $u \equiv 0$ in $(-L,L)\times (0,T')$. Using (\ref{e144}) and since $T'$ may be taken arbitrarily close to $T$, we have $u \equiv 0$ in $\R \times (0,T)$.
\end{proof}

Now we show that any weak solution of (\ref{e1}) decays exponentially to zero  in the space $L^2(\R)$.

\begin{thm}\label{explocal}
Let $a$ be a $C^2(\R)$ function satisfying (\ref{e2}) with $ 1 \leq p < 5$ and $b$ satisfying (\ref{hyp a}). Then, the system  (\ref{e1}) is locally uniformly exponentially stable in $L^2(\R)$, i.e, for any $r > 0$ there exist two constants $C > 0$ and $\eta =\eta(r) > 0$, such that, for any $u_0 \in L^2(\R)$, with $\|u_0\|_{L^2(\R)} < r$, and any weak solution $u$ of (\ref{e1}),
\begin{equation*}\label{e151}
\|u(t)\|_{L^2(\R)} \leq C\|u_0\|_{L^2(\R)}e^{-\eta t}, \quad t\geq 0.
\end{equation*}
\end{thm}

\begin{proof}
First, note that the corresponding solution $u$ of (\ref{e1}) satisfies the following estimate
\begin{equation}\label{e152}
\|u(T)\|_2^2 +2 \|u_x\|_{L^2(0,T;L^2(\R))}^2+2\int_0^T\int_{\R}b(x)|u(x,t)|^2dxdt = \|u_0\|_2^2.
\end{equation}
On the other hand, multiplying the equation in (\ref{e1}) by $(T-t)u$ and integrating on $\R\times [0,T]$, we obtain
\begin{equation}\label{e153}
\frac{T}{2} \|u_0\|_2^2= \frac{1}{2}\|u\|_{L^2(0,T;L^2(\R))}^2+\int_0^T\int_{\R}(T-t)|u_x(x,t)|^2dxdt +\int_0^T\int_{\R}(T-t)b(x)|u(x,t)|^2dxdt,
\end{equation}
which implies that
\begin{equation}\label{e154}
\|u_0\|_2^2\leq \frac{1}{T}\|u\|_{L^2(0,T;L^2(\R))}^2+2\|u_x\|_{L^2(0,T;L^2(\R))}^2 +2\int_0^T\int_{\R}b(x)|u(x,t)|^2dxdt.
\end{equation}
\begin{claim}\label{claim1}
For any $T>0$ and $r>0$ there exist $C=C(r,T)$, such that, for any weak solution $u$ of (\ref{e1}) with $\|u_0\|_2\leq r$, the following estimate holds:
\begin{equation*}\label{e155}
\|u_0\|_2^2 \leq C\left( \|u_x\|_{L^2(0,T;L^2(\R))}^2+\int_0^T\int_{\R}b(x)|u(x,t)|^2dxdt\right).
\end{equation*}
\end{claim}
%------------
\begin{proof}
By (\ref{e154}) it is sufficient to prove that
\begin{equation}\label{e156}
\|u\|_{L^2(0,T;L^2(\R))}^2 \leq C_1\left( \|u_x\|_{L^2(0,T;L^2(\R))}^2+\int_0^T\int_{\R}b(x)|u(x,t)|^2dxdt\right)
\end{equation}
provided that $\|u_0\|_2 \leq r$. In fact, we argue by contradiction and suppose that (\ref{e156}) does not hold. Hence, there exist a sequence $\{u_n\}$ of weak solution in $C_w([0,T];L^2(\R))\cap L^2(0,T;H^1(\R))$ satisfying
\begin{equation*}
\|u_n(0)\|_2 \leq r
\end{equation*}
and such that
\begin{equation*}\label{e157}
\lim_{n \rightarrow \infty}\frac{\|u_n\|_{L^2(0,T;L^2(\R))}^2}{ \|\partial_xu_n\|_{L^2(0,T;L^2(\R))}^2+\int_0^T\int_{\R}b(x)|u_n|^2dxd}=+\infty.
\end{equation*}
Note that, from (\ref{e153}), we get
\begin{equation*}\label{ln}
\lambda_n := \|u_n\|_{L^2(0,T;L^2(\R))} \leq T^{\frac{1}{2}} \|u_n(0)\|_2 \leq T^{\frac{1}{2}} r,
\end{equation*}
which implies the following convergences
\begin{align}\label{e158}
\lim_{n \rightarrow \infty} \|\partial_xu_n\|_{L^2(0,T;L^2(\R))}^2=0 \quad \text{and} \quad \lim_{n \rightarrow \infty}\int_0^T\int_{\R}b(x)|u_n|^2dxd =0.
\end{align}
Moreover, from  (\ref{ln}) we obtain a subsequence, denoted by the same index $n$, and $\lambda \geq 0$, such that
\begin{equation*}
\lambda_n \rightarrow \lambda.
\end{equation*}
Define $v_n(x,t):=\frac{u_n(x,t)}{\lambda_n}$. Then, $v_n$ is a weak solution of
\begin{equation*}\label{e159}
\left\lbrace \begin{tabular}{l}
$\partial_t v_n + \partial_x^3v_n- \partial_x^2v_n+a(\lambda_nv_n)\partial_xv_n+bv_n =0$ \\
$\|v_n\|_{L^2(0,T;L^2(\R))}=1$ \\
$v_n(x,0)=\frac{u_n(x,0)}{\lambda_n}.$
\end{tabular}\right.
\end{equation*}
Note that
\begin{equation*}
|a(\lambda_n \mu )|\leq C'(1+|\mu|^p)
\end{equation*}
and $v_n(x,0)$ is bounded in $L^2(\R)$. In fact, by (\ref{e154}) and (\ref{e158}) we obtain
\begin{equation}\label{vn0}
\|v_n(0)\|_2^2 \leq \frac{1}{T}+\frac{2}{\lambda_n^2}\left\lbrace \|\partial_x u_n\|_{L^2(0,T;L^2(\R))}^2 +\int_0^T\int_{\R}b(x)|u_n(x,t)|^2dxdt \right\rbrace
\end{equation}
Combining (\ref{e158}), (\ref{vn0}) and (\ref{e152}) we conclude that $\{v_n\}$ is bounded in $L^{\infty}(0,T;L^2(\R))\cap L^2(0,T;H^1(\R))$.
Hence, extracting a subsequence if needed, we have
\begin{equation*}
u_n \rightharpoonup u \quad \text{in $L^{\infty}([0,T];L^2(\R))$ weak $*$},
\end{equation*}
\begin{equation*}
u_n \rightharpoonup u \quad \text{in $L^{2}([0,T];H^1(\R))$ weak},
\end{equation*}
as $n \rightarrow \infty$. In order to analyze the nonlinear term, we consider the function
\begin{equation*}
A(v):=\int_0^va(\lambda u)du, \qquad A_n(v):=\int_0^va(\lambda_nu)du.
\end{equation*}
Proceeding as in the proof of Theorem \ref{teo4}, it is easy to see that $a(\lambda_n v_n)\partial_x v_n=\partial_x[A_n(v_n)]$ is bounded in $L^{\alpha}([0,T];H^{-2}_{loc}(\R))$, for  $\alpha \in \left(1, \frac{6}{p+1}\right)$, and  $\partial_t v_n=- \partial_x^3v_n+\partial_x^2v_n-a(\lambda_nv_n)\partial_xv_n-bv_n$ is bounded in $L^{\alpha}([0,T];H^{-2}_{loc}(\R)) \hookrightarrow L^{1}(0,T;H^{-2}_{loc}(\R))$. Since $\{v_n\}$ is bounded in $L^{2}([0,T];H^1(\R))$, we can obtain a subsequence, such that
\begin{equation}\label{strong}
\text{$v_n \rightarrow v$ strong in $L^{2}(\R \times (0,T)) \equiv L^2(0,T;L^2(\R))$}
\end{equation}
and
\begin{equation*}\label{e161}
\text{$a(\lambda_nv_n)\partial_xv_n \longrightarrow a(\lambda v)\partial_x v$ in $D'(\R \times [0,T])$},
\end{equation*}
where $v$ solves the equation
\begin{equation*}
v_t+v_{xxx}-v_{xx}+a(\lambda v)v_x+bv=0 \quad \text{in $D'([0,T]\times \R)$}.
\end{equation*}
On the other hand,  by (\ref{e158}) and (\ref{strong}), we obtain
\begin{equation}\label{strong1}
\|v\|_{L^2(0,T;L^2(\R))}=1\quad  \text{and}\quad \lim_{n\rightarrow \infty }\int_0^T\int_{\R}b(x)|v_n(x,t)|^2dxdt=0.
\end{equation}
Consequently, due to (\ref{hyp a}) it follows that
\begin{equation*}
\text{$v \equiv 0$ on $[0,T]\times\omega$}.
\end{equation*}
%-----------
\begin{claim}\label{claim2}
Let $0 < t_1 < t_2 < T$. Then, there exists $(t'_1,t'_2) \subset (t_1, t_2)$, such that $v \in L^{\infty}(t'_1,t'_2; H^1(\R))$
\end{claim}
%-------------
\begin{proof}
Let $w_n$ be solution of
\begin{equation*}
\left\lbrace
\begin{tabular}{l l}
$\partial_t w_n-\partial_x^2 w_n +\partial_x^3 w_n + a_n(\lambda_n w_n)\partial_x w_n = 0$ & in $\R\times (0,T)$ \\
$w_n(x,0)=v_n(x,0)$  & in $\R$
\end{tabular}\right.
\end{equation*}
where $a_n \in C_0^{\infty}(\R)$ satisfies (\ref{e41}) and (\ref{e42}). Proceeding as in the proof of the Theorem \ref{teo4}, we have that
\begin{equation}\label{e162}
\text{ $w_n-v_n \rightarrow 0$ in $C([0,T];H^{-1}_{loc}(\R))$} \quad  \text{and} \quad  \|w_n\|_{L^2(0,T;H^1(\R))}\leq C.
\end{equation}
Consider $\tau_n \in \left( t_1, \frac{t_1+t_2}{2}\right)$, such that
\begin{align*}
& \text{$\tau_n \rightarrow \tau$ and $\|w_n(\tau_n)\|_{L^2(0,T;H^1(\R))} \leq C$.}
\end{align*}
Hence, by Theorem \ref{teo5},
\begin{equation}\label{e163}
 \|w_n(\tau_n+\cdot)\|_{L^2(0,\varepsilon;H^1(\R))} \leq C,
\end{equation}
for any $\varepsilon \leq T$. On the other hand, note that (\ref{e162}) implies that
\begin{equation}\label{e164}
 w_n(\tau_n +\cdot) \rightarrow v(\tau + \cdot) \quad \text{in $C([0,\varepsilon];H^{-1}_{loc}(\R))$}
\end{equation}
for $\varepsilon < \frac{t_2-t_1}{2}$. Thus by (\ref{e163}) and (\ref{e164}), $v \in L^{\infty}(\tau,\tau+\varepsilon; H^1(\R))$.
\end{proof}
Applying the claim above and Lemma \ref{ucp}, we deduce that $v = 0$ in $\R \times (t'_1,t'_2)$, where $(t'_1,t'_2)\subset (t_1,t_2)$.
As $t_2$ can be arbitrary close to $t_1$, we obtain by continuity of $v$ in $H^{-1}_{loc}(\R)$ that $v \equiv 0$, which contradicts (\ref{strong1}).
\end{proof}
By claim \ref{claim1} and (\ref{e152}), the following estimate holds
\begin{equation*}
\|u(T)\|_{L^2(\R)}^2 \leq \gamma \|u_0\|_{L^2(\R)}^2, \quad \text{with $0< \gamma <1$}.
\end{equation*}
Consequently,
\begin{equation*}
\|u(kT)\|_{L^2(\R)}^2 \leq \gamma^k \|u_0\|_{L^2(\R)}^2, \quad \forall k\geq 0.
\end{equation*}
Moreover, for any $t \geq 0$, there exist $k>0$, such that $kT \leq t < (k+1)T$. Thus,
\begin{align*}
\|u(t)\| _{L^2(\R)}^2 &\leq \|u(kT)\|_{L^2(\R)}^2 \leq \gamma^k \|u_0\|_{L^2(\R)}^2 \\
& \leq \gamma^{\frac{t}{T}}\gamma^{-1} \|u_0\|_{L^2(\R)}^2 \\
&\leq \gamma^{-1}\|u_0\|_{L^2(\R)}^2e^{-\eta t},
\end{align*}
where $\eta = -\frac{ln \gamma}{T} > 0$.
\end{proof}

The next result asserts that the system (\ref{e1}) is globally uniformly exponentially stable in $L^2(\R)$. It means that the constant $\eta$  in Proposition \ref{explocal} is independent of $r$, when $\|u_0\|_{L^2(\R)} \leq r$.

\begin{thm}\label{globalexp}
Let $a$ be a $C^2(\R)$ function satisfying (\ref{e2}), with $ 1 \leq p < 5$, and $b$ satisfying (\ref{hyp a}). Then, the system  (\ref{e1}) is globally uniformly exponentially stable in $L^2(\R)$, i.e,  there exist a positive constant $\eta$ and a nonegative continuous function $\alpha: \R \rightarrow \R$, such that, for any $u_0 \in L^2(\R)$ with $\|u_0\|_{L^2(\R)} < r$ and any weak solution $u$ of (\ref{e1}),
\begin{equation*}\label{e165}
\|u(t)\|_{L^2(\R)} \leq \alpha(\|u_0\|_{L^2(\R)})e^{-\eta t}, \quad t\geq \varepsilon,
\end{equation*}
for all $\varepsilon >0$, where $\eta=\eta(\varepsilon)$.
\end{thm}
%;--------------------------------------------------------------
\begin{proof}
By Proposition \ref{explocal}, there exist $\eta '= \eta'(\varepsilon)>0$ and $C=C(\varepsilon)>0,$ such that
\begin{equation*}
\|u(t)\|_{L^2(\R)} \leq C\|u(\varepsilon)\|_{L^2(\R)}e^{-\eta' t}, \quad t\geq \varepsilon.
\end{equation*}
If $\|u_0\|_{L^2(\R)}\leq r$, again, by Proposition \ref{explocal} there exist $C_r >0$ and $\eta_r >0$, satisfying
\begin{equation*}
\|u(t)\|_{L^2(\R)} \leq C_r\|u_0\|_{L^2(\R)}e^{-\eta_r t}, \quad t\geq 0.
\end{equation*}
Thus, for all $t\geq \varepsilon$, we have
\begin{align*}
\|u(t)\|_{L^2(\R)} &\leq C\|u(\varepsilon)\|_{L^2(\R)}e^{-\eta' t} \\
& \leq CC_r\|u_0\|_{L^2(\R)}e^{-\eta_r \varepsilon}e^{-\eta' t} \\
& \leq \alpha(\|u_0\|_{L^2(\R)})e^{-\eta' t},
\end{align*}
where $\alpha(s)=CC_re^{-\eta_r \varepsilon} s$.
\end{proof}

\section{Appendix}

\begin{proof}[Proof Lemma \ref{carl}]
We follow the steps of \cite{rosier2000exact}. Let $q=q(x,t)$ satisfying (\ref{e131}) and (\ref{e132}) and  $\varphi(t,x)=\frac{\psi(x)}{t(T-t)}$,  where $\psi$ is a positive function (to be specified later). Consider $u := e^{-s\varphi}q$ and $w := e^{-s\varphi}P(e^{s\varphi} u)$, where $P$ is a differential operator given by
\begin{equation*}
P=\partial_t-\partial^2_x +\partial_x^3.
\end{equation*}
Note that
\begin{align*}
&\partial_t(e^{s\varphi}u)=e^{s\varphi}\left\lbrace s\varphi_tu+u_t\right\rbrace \\
&\partial_x(e^{s\varphi}u)=e^{s\varphi}\left\lbrace s\varphi_xu+u_x\right\rbrace \\
&\partial_x^2(e^{s\varphi}u)=e^{s\varphi}\left\lbrace s\varphi_{xx}u+s^2\varphi_{x}^2u+2s\varphi_{x}u_{x}+u_{xx}\right\rbrace \\
&\partial_x^3(e^{s\varphi}u)=e^{s\varphi}\left\lbrace s\varphi_{xxx}u+3s^2\varphi_{x}\varphi_{xx}u+3s\varphi_{xx}u_{x}+s ^3\varphi_{x}^3u+3s^2\varphi_{x}^2u_{x}+3s\varphi_{x}u_{xx}+u_{xxx}\right\rbrace.
\end{align*}
Hence,
\begin{align*}
P(e^{s\varphi}u)=e^{s\varphi}& \left\lbrace \left( s\varphi_t+s\varphi_{xxx}+3s^2\varphi_{x}\varphi_{xx}+s^3\varphi_{x}^3-s\varphi_{xx}-s^2\varphi_{x}^2\right)u+\left(3s\varphi_{xx}+3s^2\varphi_{x}^2-2s\varphi_{x}\right) u_x \right. \\
&\left.+\left(3s\varphi_{x}-1\right)u_{xx}+u_{xxx}+u_t\right\rbrace
\end{align*}
and
\begin{equation}\label{e133}
w=Au+Bu_{x}+Cu_{xx}+u_{xxx}+u_{t},
\end{equation}
where
\begin{align*}
A &= s(\varphi_t-\varphi_{xx}+\varphi_{xxx})+3s^2\varphi_{x}\varphi_{xx}+s^3\varphi_{x}^3-s^2\varphi_{x}^2  \\
B &= 3s\varphi_{xx}+3s^2\varphi_{x}^2-2s\varphi_{x} \\
C &= 3s\varphi_{x}-1.
\end{align*}
Set $L_1u:=u_t+u_{xxx}+Bu_{x}$ and $L_2 u:=Au+Cu_{xx}$. Thus, we have
\begin{equation}\label{e4.6}
2\int_0^T\int_{-L}^{L}L_1(u)L_2(u)dxdt\leq \int_0^T\int_{-L}^{L}\left( L_1(u)+L_2(u)\right)^2 dxdt\ = \int_0^T\int_{-L}^{L}w^2dxdt
\end{equation}
Next, we compute the double product in (\ref{e4.6}). Let us denote by $(L_iu)_j$ the $j$-th term of $L_iu$ and $Q=(0,T)\times(-L,L)$. Then, to compute the integrals on the right hand side of (\ref{e4.6}), we perform integrations by part in $x$ or $t$:
\begin{align*}
\left( (L_1u)_1,(L_2u)_1\right)_{L^2(Q)}&=-\frac{1}{2} \int_Q A_tu ^2dxdt \\
\left( (L_1u)_2,(L_2u)_1\right)_{L^2(Q)}&=-\frac{1}{2} \int_Q A_{xxx}u^2dxdt + \frac{3}{2} \int_Q A_x u_x^2dxdt \\
\left( (L_1u)_3,(L_2u)_1\right)_{L^2(Q)}&- \frac{1}{2} \int_Q (AB)_{x}u^2dxdt \\
\left( (L_1u)_2,(L_2u)_2\right)_{L^2(Q)}&= -\frac{1}{2} \int_Q C_{x}u^2dxdt \\
\left( (L_1u)_3,(L_2u)_2\right)_{L^2(Q)}&= -\frac{1}{2} \int_Q (BC)_{x}u^2_xdxdt
\end{align*}
By (\ref{e133}) and Young inequality, it follows that
\begin{align*}
\left( (L_1u)_1,(L_2u)_2\right)_{L^2(Q)}&=\frac{1}{2}\int_Q C_tu^2_xdxdt +\frac{1}{2}\int_Q AC_x \partial_x(u^2)dxdt + \int_Q BCu_x^2dxdt+\frac{1}{2}\int_Q CC_x \partial_x(u^2_x)dxdt\\
&+\int_Q C_xu_xu_{xxx}dxdt-\int_Q C_x u_x wdxdt \\
&\geq \frac{1}{2}\int_Q \left\lbrace C_t+2BC_x -(CC_x)_x+C_{xxx}\right\rbrace u^2_xdxdt-\frac{1}{2}\int_Q (AC_x)_x u^2dxdt - \int_Q C_xu_{xx}^2dxdt \\
&-\varepsilon\int_Q C_x^2 u_x^2 dxdt-C(\varepsilon)\int_Q w^2 dxdt,
\end{align*}
where $\varepsilon$ is any number in $(0,1)$. Putting together the computations above, we obtain
\begin{equation}\label{e137}
\int_Q \left\lbrace D u^2 +Eu_x^2+ F u_{xx}^2\right\rbrace dxdt = 2 \left( (L_1u,L_2u\right)_{L^2(Q)} \leq C(\varepsilon) \int_Q w^2dxdt
\end{equation}
with, $D$, $E$ and $F$ given by
\begin{align}
&D = - \left( A_t+A_{xxx}+(AB)_x+(C_xA)_x\right) \notag \\
&E= 3A_x +BC_x-B_xC- (CC_x)_x+C_xxx +C_t -\varepsilon C_x^2 \label{eE}\\
&F=-3C_x.\notag
\end{align}
The identities above allow us to conclude that
\begin{align*}
D &= -15 s^5\varphi_{x}^4\varphi_{xx}+\frac{O(s^4)}{t^4(T-t)^4}, \qquad \text{(as $s\rightarrow \infty$)}, \\
& = -15 \frac{s^5}{t^5(T-t)^5}\psi_{x}^4(x)\psi_{xx}(x)+\frac{O(s^4)}{t^4(T-t)^4}.
\end{align*}
If we consider
\begin{equation}\label{c1}
\text{$|\psi_x(x)|>0$ and $\psi_{xx}(x) < 0$, for all $x \in [-L,L]$},
\end{equation}
taking $s$ large enough, we obtain a constant $C_1>0$, such that
\begin{equation}\label{cD}
D \geq C_1 \frac{s^5}{t^5(T-t)^5}.
\end{equation}
On the other hand, note that
\begin{align*}
BC_x = & 9s^2\varphi_{xx}^2+9s^3\varphi_{x}^2\varphi_{xx}-6s^2\varphi_{x}\varphi_{xx} \\
B_xC= & 9s^2\varphi_{x}\varphi_{xxx}+18s^3\varphi_{x}^2\varphi_{xx}-12s^2\varphi_{x}\varphi_{xx}-3s\varphi_{xxx}+2s\varphi_{xx} \\
3A_x = &3s\varphi_{xt}+3s\varphi_{4x}-3s\varphi_{xxx}+9s^2\varphi_{xx}^2+9s^2\varphi_{x}\varphi_{xxx}+9s^3\varphi_{x}^2\varphi_{xx}-6s^2\varphi_{x}\varphi_{xx}\\
CC_x = &9s^2\varphi_{x}\varphi_{xx}-3s\varphi_{xx} \\
(CC_x)_x = &9s^2\varphi_{xx}^2+9s^2\varphi_{x}\varphi_{xxx}-3s\varphi_{xxx} \\
C_{xxx}+C_t-\varepsilon C_x^2 = &3s\varphi_{4x}+3s\varphi_{xt}-9\varepsilon s^2\varphi_{xx}^2
\end{align*}
Putting together these expressions, we have
\begin{align*}
E = 6s\varphi_{xt}+6s\varphi_{4x}+9(1-\varepsilon)s^2\varphi_{xx}^2-9s^2\varphi_{x}\varphi_{xxx}+3s\varphi_{xxx}-2s\varphi_{xx},
\end{align*}
for $E$ defined in (\ref{eE}). Hence,
\begin{align*}
E = & 9s^2 \left\lbrace (1-\varepsilon)\varphi_{xx}^2-\varphi_x\varphi_{xxx}\right\rbrace +\frac{O(s)}{t^2(T-t)^2}, \qquad \text{as $s \rightarrow \infty$,} \\
= &9\frac{s^2}{t^2(T-t)^2} \left\lbrace (1-\varepsilon)\psi_{xx}^2(x)-\psi_x(x)\psi_{xxx}(x)\right\rbrace +\frac{O(s)}{t^2(T-t)^2}.
\end{align*}
For $s$ large enough and $\psi$ satisfying
\begin{equation}\label{c2}
\psi_x(x)\psi_{xxx}(x) < (1-\varepsilon)\psi_{xx}^2(x), \quad \text{for all $x\in [-L,L]$},
\end{equation}
we obtain a constant $C_2 >0$, such that
\begin{equation}\label{cE}
E \geq C_2\frac{s^2}{t^2(T-t)^2}.
\end{equation}
Finally, since
\begin{align*}
F= -9s\varphi_{xx} = - \frac{9s\psi_{xx}(x)}{t(T-t)},
\end{align*}
(\ref{c1}) guarantees the existence of a constant $C_3 >0$, such that
\begin{equation}\label{cF}
F\geq C_3 \frac{s}{t(T-t)}.
\end{equation}
Combining (\ref{cD}), (\ref{cE}), (\ref{cF}) and (\ref{e137}), we obtain
\begin{equation*}
\int_Q \left\lbrace \frac{s^5}{t^5(T-t)^5}|u|^2+\frac{s^2}{t^2(T-t)^2}|u_x|^2+\frac{s}{t(T-t)}|u_{xx}|^2\right\rbrace dxdt \leq C_4 \int_ Q w^2dxdt,
\end{equation*}
for some $C_4 >0$. On the other hand, note that
\begin{align*}
\int_Q \frac{s^3}{t^3(T-t)^3}u_x^2dxdt &= -\int_Q\frac{s^3}{t^3(T-t)^3}uu_{xx}dxdt \leq \int_Q \frac{s^5}{2t^5(T-t)^5}|u|^2dxdt +\int_Q\frac{s}{2t(T-t)}|u_{xx}|^2dxdt \\
&\leq \frac{C_4}{2}\int_Q w^2 dxdt.
\end{align*}
Then,
\begin{equation*}
\int_Q \left\lbrace \frac{s^5}{t^5(T-t)^5}|u|^2+\frac{s^3}{t^3(T-t)^3}|u_x|^2+\frac{s}{t(T-t)}|u_{xx}|^2\right\rbrace dxdt \leq C_5 \int_ Q w^2dxdt
\end{equation*}
where $C_5 >0$, provided that (\ref{c1}) and (\ref{c2}) hold. Returning to the original variable $u=e^{-s\varphi}q$, we conclude the proof of Lemma \ref{carl}.
\end{proof}

\end{document}